\documentclass[12pt]{amsart}
\usepackage{amsmath}
\usepackage{amsthm}
\usepackage{amsfonts}
\usepackage{amscd}
\usepackage{amssymb}
\usepackage{amstext}
\usepackage{hyperref}
\usepackage{epsfig}
\usepackage{subfiles}
\usepackage{xcolor}
\usepackage{mathrsfs}

\numberwithin{equation}{section}

\allowdisplaybreaks

\newtheorem{theorem}{Theorem}[section]
\newtheorem{proposition}{Proposition}[section]
\newtheorem{lemma}{Lemma}[section]
\newtheorem{remark}{Remark}[section]
\newtheorem{definition}{Definition}[section]
\newtheorem{corollary}{Corollary}[section]

\renewcommand{\epsilon}{\varepsilon}

\newcommand{\abs}[1]{\left\vert #1\right\vert}

\newcommand{\R}{\mathbb{R}}
\newcommand{\N}{\mathbb{N}}
\newcommand{\Z}{\mathbb{Z}}
\newcommand{\T}{\mathcal{T}} 
\newcommand{\PR}{\mathbb{P}}

\newcommand{\reff}{R_{\text{eff}}}

\title[]{Scaling limit for the random walk on critical lattice trees}
\date{}
\author[G.~Ben Arous]{G\'erard Ben Arous}
\address{G\'erard Ben Arous, Courant Institute of Mathematical Sciences, 251 Mercer Street, New York University,
New York, 12012-1185, U.S.A.} \email{benarous@cims.nyu.edu}
\author[M.~Cabezas]{Manuel Cabezas}
\address{Manuel Cabezas\\ Pontificia Universidad Cat\'{o}lica de Chile, Facultad de Matem\'{a}ticas, Campus San Joaqu\'{i}n, Avenida Vicu\~{n}a Mackenna 4860, Santiago, Chile.} \email{mncabeza@mat.uc.cl}

\author[A.~Fribergh]{Alexander Fribergh}
\address{Alexander Fribergh\\Universit\'e de Montr\'eal, DMS\\
Pavillon Andr\'e-Aisenstadt\\     2920, chemin de la Tour Montréal (Qu\'ebec),  H3T 1J4} \email{fribergh@dms.umontreal.ca}


\begin{document}

\begin{abstract}
We prove a scaling limit theorem for the simple random walk on critical lattice trees in $\Z^d$, for $d\geq 8$. The scaling limit is the Brownian motion on the Integrated Super-Brownian Excursion (BISE) which is the same one that we have identified earlier for other simpler models of anomalous diffusion on critical graphs in large enough dimension. The proof of this theorem is based on a combination of the tools of lace-expansion (contained in the articles \cite{CFHP} and \cite{CFHP2}), and a new and general convergence theorem. \end{abstract}

\maketitle

\section{Introduction}

The idea that diffusion is expected to be anomalous on critical graphs goes back at least to Pierre-Gilles de Gennes' work in 1976 on the ``ant in the labyrinth", i.e., the random walk on critical percolation clusters, which has been followed by a massive effort in the physics literature to gain insight in the matter (see for instance the classical  works by ~\cite{rammal1983random}, ~\cite{HavlinBenAvraham}, or for a more recent work  ~\cite{Biroli2019}). The mathematics literature has also been very active (since the classical work of  ~\cite{kesten1986subdiffusive} proving sub-diffusivity for the case of critical percolation in dimension 2).  An important milestone has been reached by Kozma and Nachmias in \cite{KN}, where they proved the Alexander-Orbach conjecture, finally establishing that the spectral dimension of the critical percolation cluster in high enough dimension is $\frac{4}{3}$. 

This article is part of our effort  to understand in more detail anomalous diffusion on critical graphs in large enough dimension, where the behavior is supposed to be universal (see for instance  \cite{Barlow_Kumagai}, \cite{Croydon_crt}, \cite{HHH}, \cite{arous2016scaling} or the survey \cite{heydenreich2017progress}). 

We aim at a detailed result i.e., understanding the full scaling limit of random walks on these critical graphs. The natural candidate as the scaling limit is the so-called \emph{BISE} (the Brownian motion on the Integrated Super-Brownian Excursion, see Section~\ref{sect_bise_def} for a formal definition). In our first step ~\cite{BCFa} we gave an abstract theorem ensuring that the scaling limit would indeed be the BISE under a set of four conditions (see Section~\ref{sect_conditions} and Section~\ref{sect_thms}), which we assume would be satisfied in this class of critical random graphs in large enough dimension.

As a first proof of concept, this theorem was applied in~\cite{BCFb} in the context of the simple random walk on the range of critical branching random walks (conditioned on being large). But, in order to deal with more delicate models it appeared that one would have to enlist the heavier tools of the field, i.e. lace expansion, in order to check the four conditions of ~\cite{BCFa}. 

Our goal is now to show that a similar general result can be applied to the entire class of models that can be analyzed using lace expansion. These models include lattice trees above dimension 8 (see e.g.~\cite{haraslade,derbezslade,holmes2008convergence}), critical oriented percolation above dimension 4 (see e.g.~\cite{HHS}) and critical percolation above dimension 6 (see e.g.~\cite{harasladeperc}). Our final goal is to understand the case of percolation and oriented percolation, which is the initial question raised by De Gennes, albeit in high dimension. We start this project here with the case of (spread-out) lattice trees. It is important to note here that this article relies indeed on lace-expansion estimates obtained in the two related articles~\cite{CFHP} and~\cite{CFHP2}.

In the process of building a proof in the context of lattice trees, we observed that under certain very natural assumptions (which are valid for lattice trees), the four conditions needed for convergence towards the BISE can actually essentially be reduced to three. This represents the main work of this article.

We start by presenting the model of the simple random walk on spread-out lattice trees and stating our main convergence result in this context (see Theorem~\ref{main_thm}) since it requires less notations than the general case. We then introduce the notations needed to state our new convergence theorem for critical random graphs in general which will lead up to our new conditions for convergence towards the BISE in Section~\ref{sect_thms}. We then move onto the main proof before applying our results to the lattice tree case in Section~\ref{sect_lattice_trees}.

\subsection{Critical lattice trees}


A lattice tree $T$ is a finite, connected subgraph of $\mathbb{Z}^d$ containing no cycles.  Specifically, we will be considering \emph{spread-out} lattice trees, i.e., subgraphs of $(\Z^d,E_L(\Z^d))$, where \[E_L(\Z^d):=\{[x,y]:x,y\in\Z^d, x\neq y,\|x-y\|_\infty\leq L\}\] is the set of edges of $\Z^d$ whose endpoints are at $\|\cdot\|_\infty$-distance smaller or equal than some constant $L\in\N$. 
 Let us now define a probability distribution on lattice trees. Let $D(\cdot)$ be uniform distribution on a finite box $[-L,L]^d\setminus o$, where $o=(0,\dots,0)\in \Z^d$. That is $D(x)=\frac{1}{(2L)^d-1}$ for $x\in[-L,L]^d\setminus o$ and $0$ otherwise. For a lattice tree $T\ni o$ and $z>0$, define the weight of $T$ as 
 \begin{equation}
 W_{z,D}(T):=z^{|E(T)|}\prod_{[x,y]\in E(T)}D(x-y)=\left(\frac{z}{(2L)^d-1}\right)^{|E(T)|},
 \end{equation} where $E(T)$ is the set of edges of $T$ and $|E(T)|$ is its cardinal.  An important observation is that if $T$ is an edge-disjoint union of subtrees, then $W_{z,D}(T)$ can be factored into a product over the weights of the subtrees.  It turns out (see e.g.~\cite{klein1981rigorous}) that there exists a critical value $z_D$ such that $\rho:=\sum_{T\ni o}W_{z_D,D}(T)<\infty$ and ${\bf E}[|\T|]=\infty$, where ${\bf E}$ denotes expectation with respect to the probability measure ${\bf P}$ defined as ${\bf P}[\T=T]=\rho^{-1}W_{z_D,D}(T)$.  Hereafter we write $W(\cdot)$ for the critical weighting $W_{z_D,D}(\cdot)$ and suppose that we are selecting a random tree $\T$ according to this critical weighting.

We denote $d_T(\cdot,\cdot)$ the graph distance of a lattice tree $T$. We can then define the height of a lattice tree $T$ as
\[
H(T)=\max\{d_T(o,x) \text{ with } x\in T\}.
\]

Let $h>0$ and we set ${\bf P}^h_{n}[\hspace{0.1 cm}\cdot\hspace{0.1 cm}]:={\bf P}[\hspace{0.1 cm}\cdot\mid H(T)\geq hn]$ and denote $T_n$ a random tree chosen according to that measure.

Given a  fixed lattice tree $T$, we are able to define the simple random walk started at $o$ on that tree.  We denote this process $(X_k^T)_{k\geq 0}$ and its corresponding law $P^{T}$.

In this article, we will give scaling limit results on this random walk under the annealed measure $\PR_n^h[\cdot]={\bf E}_n^h[P^{T_n}[\cdot]]$. More precisely, our main result on lattice trees is the following.
\begin{theorem}\label{main_thm}
Fix $d>8$. There exists $L_0(d)$ such that for all $L\geq L_0$ there exists constants $C_1, C_2 \in (0,\infty)$ such that, for each $h>0$, 
\[
(n^{-1/2}X^{T_n}_{n^3t})_{t \geq 0}\stackrel{n\to\infty}{\to} (C_1 B^{h\text{-ISE}}_{C_2 t})_{t\geq 0},
\]
where the convergence is annealed (under  $\PR^h_{ n}$) and occurs with the topology of uniform convergence over compact sets of time. The limiting process $B^{h\text{-ISE}}$ is the Brownian motion on the \emph{integrated super-Brownian excursion} conditioned to have survived for time at least $h$ and is defined in Proposition~\ref{propdef_BISE}.
\end{theorem}

\subsection{Simple conditions for the convergence towards the \emph{Brownian motion on the ISE}}\label{sect_intro_cond}

Let us now discuss our theorem for simple random walks on general critical graphs in high dimensions, without going into full details because the notations needed for the precise statement are quite involved. 

In the initial paper~\cite{BCFa}, we established a very general framework for a sequence $G_n$ of random subgraphs of $\Z^d$, within which one can prove convergence of the simple random walk on the graphs towards the BISE. In the article \cite{BCFb}, that framework was applied to prove convergence of simple random walks on critical branching random walks in high dimensions towards the  BISE.

This type of convergence result is proved by verifying four conditions called $(G)$, $(V)$, $(R)$ and $(S)$. Those conditions depend on the choice of certain random points $(V_i)_{i\in\N}=(V_i^n)_{i\in\N}$ on the graphs $G_n$, which are used to span subtrees from $G_n$. More precisely, for each $K\in\N$, we will consider the \emph{skeleton} $\T^{(n,K)}$ as the minimal subgraph of $G_n$ containing $o$ and $V_i,i=1,\dots, K$(for more details see Section~\ref{sect_conditions}).  

Our main contribution in this article is that, for a specific choice of the random spanning points $(V_i)_{i\in\N}$,  one of the four conditions simplifies substantially. Our assumption on the points
$(V_i)_{i\in\N}$ is that they are chosen \emph{uniform according to edge-volume}, which, simply put, says that the distribution of each $V_i$ is asymptotically equal to the one given by first choosing an edge uniformly at random, and then to choose one of the end-vertices of the edge, uniformly and independently.  The formal statement of the main result can be found at Theorem~\ref{thm:height2}.


Let us describe informally the conditions appearing in this improved theorem.
\begin{itemize}
\item Condition $(G)$: The geometric convergence of the (rescaled) skeletons $\T^{(n,K)}$ towards the skeletons taken on the integrated super-Brownian excursion.  This condition was proved for lattice trees in~\cite{CFHP}, using the \emph{historical measure-valued process} convergence to the super Brownian Motion obtained in ~\cite{CFHP2}. We believe that the argument used in~\cite{CFHP} for lattice trees could be carried over to other models commonly studied using lace expansion techniques such as critical percolation.
\item Condition $(R)$: The asymptotic proportionality between the effective electrical resistance between two vertices of $G_n$ and the graph distance between these two vertices. Since the effective resistance governs the transition probabilities of a simple random walk, this second condition ensures that the random walk has (asymptotically) symmetric transition probabilities. This condition holds trivially for (lattice) trees.
\item Condition $(S)$: For any $\epsilon>0$, there exists a $K$ such that for $n$ large enough the graphs $G_n$ (properly rescaled) are within an $\epsilon$-neighborhood of the skeleton $\T^{(n,K)}$ in the Euclidean distance and also in the intrinsic distance (the article \cite{CFHP} shows why this result holds in the case of lattice trees). 
\end{itemize}

Let $E(G_n)$ denote the set of edges of $G_n$. Our main result (expressed in two different ways, corresponding to different conditionings in Theorems \ref{thm:height2} and \ref{thm:volume}) is that, for a sequence of random graphs together with random points  $(G_n,(V_i^n)_{i\in\N})_{n\in\N}$, where the $(V_i^n)_{i\in\N}$ are chosen according to uniform edge volume, it holds that  
\begin{enumerate}
\item Under conditions $(G)$, $(S)$ and $(R)$, we have
\[
(n^{-1/2}X^{G_n}_{n\abs{E(G_n)}t})_{t \geq 0}\stackrel{n\to\infty}{\to} (C_1 B^{\text{ISE}}_{c_2t})_{t\geq 0}
\]
where $C_1,c_2$ are positive constants.
\item If, in addition, there exists $\nu>0$ such that $\frac{\abs{E(G_n)}}{n^2}$ converges to $\nu$ in probability, we have
\[
(n^{-1/2}X^{G_n}_{n^3t})_{t \geq 0}\stackrel{n\to\infty}{\to} (C_1 B^{\text{ISE}}_{\nu^{-1}c_2t})_{t\geq 0},
\]
\end{enumerate}
where the convergences are annealed and occur with the topology of uniform convergence over compact sets of time.

The contribution of this article can be roughly summarized by saying that, when the points $(V_i)_{i\in\N}$ are chosen uniform according to edge volume, condition $(V)$ can be replaced by a much weaker estimate on the the overall edge volume of the graph. Loosely speaking, Condition $(V)$ states that for $K$ large enough, the edge volume of $G_n$ is (asymptotically) distributed evenly over the skeleton $\T^{(n,K)}$ and also there exists a constant $\nu>0$ such that $\frac{\abs{E(G_n)}}{n^2}$ converges to $\nu$ in probability.

\subsection{Organization of the paper}

We start by two sections of notations necessary to state the conditions on random graphs that are relevant for proving that the random walk on them scales to the BISE.

In Section \ref{sect_bise_def}, we will introduce the Integrated super-Brownian excursion \emph{ISE} and the Brownian motion on it. 

In Section \ref{sect_abstract_theorem} we will show how to span certain subtrees from large random graphs. These subtrees, called the \emph{skeletons}, are finite dimensional and will be used to approximate the original random graphs.

 Then, in Section~\ref{sect_conditions}, we introduce all four conditions ($(G)$, $(V)$, $(R)$ and $(S)$) from~\cite{BCFa} that will be needed to state our main theorems.
  
Section \ref{sect_thms} contains the rigorous statement of the main theorems of this paper, which give the convergence of the random walks on the graphs to the  BISE.  As we have mentioned, in this article we prove that, when the spanning points are chosen \emph{according to uniform edge volume}, conditions $(G)$, $(R)$ and $(S)$ are enough to get convergence towards the BISE (see Theorems \ref{thm:height2} and \ref{thm:volume}).

The next sections (Section 6 and Section 7) are devoted to the main proof of this article which states that: when the points $(V_i)_{i\in \N}$ are chosen  according to uniform edge volume, Condition $(G)$ essentially implies Condition $(V)$ (see~Proposition \ref{prop_cond_V} in this article). Condition $(V)$ is about understanding the convergence of the uniform measure on $G_n$ (see Definition \ref{def_condV}). In Section \ref{sect_gv} we construct an empirical measure in $G_n$, which, on the one hand, is a good approximation of the uniform measure in $G_n$ and, on the other hand, is \emph{finite-dimensional} in the sense that it depends only on geometric information of the skeleton of the graph. In Section \ref{sect_gv2} we control this empirical measure using Condition $(G)$ which allows us to obtain Condition $(V)$.

Finally, in Section~\ref{sect_lattice_trees}, we will apply Theorem~\ref{thm:height2} to the case of critical lattice trees. This proves Theorem~\ref{main_thm}.

\subsection{Notations}

Let $G$ be a graph. For $A$ a finite subgraph of $G$, we denote
\[
\partial^{(in)} A=\{x\in A,\ \text{ there exists $y\in G\setminus A$ such that $x\sim_G y$}\},
\]
\[
\partial^{(ex)} A=\{x\notin A,\ \text{ there exists $y\in  A$ such that $x\sim_G y$}\},
\]
\[
\partial_E A=\{(x,y)\in E(G),\ \text{ for $x\in A$ and $y\in G\setminus A$}\}.
\]

Given a graph $G$, we will denote $V(G)$ the set of its vertices and $E(G)$ the set of its edges. For $x\in G$ and $k\in \R$, we will write $B_G(x,k)$ for the ball of radius $k$ centered at $x$ in the natural metric induced by $G$.

The constants in this paper will typically be denoted $c$ (for lower bounds) and $C$ (for upper bounds) and implicitly assumed to be positive and finite. Their value may change from line to line.

\section{Brownian motion on super-brownian motions}\label{sect_bise_def}

The goal of this section is to introduce the Brownian motion taking values in (the support of) certain versions of the super-Brownian motion. In particular, we will introduce the \emph{Continuum Random Tree} ($CRT$), the \emph{Integrated Super-Brownian Excursion} ($ISE$) and the \emph{Brownian motion on the ISE} ($B^{\text{ISE}}$) and their counterparts obtained by conditioning on height. The presentation is taken from notes of Le Gall (see~\cite{LG}).

\subsection{The continuum random tree (CRT) and random trees conditioned on height}\label{sect:crt}

Denote by $({\bf e}(t))_{0\leq t\leq 1}$ a normalized Brownian excursion. Informally, $({\bf e}(t))_{0\leq
t\leq 1}$ is just a Brownian path started at the origin and conditioned to stay positive
over the time interval $(0,1)$, and to come back to $0$ at time $1$ (see e.g.~Sections 2.9
and 2.12 of It\^o and McKean \cite{IM} for a discussion of the normalized excursion). We
extend
the definition of ${\bf e}$ by setting ${\bf e}(t)=0$ if $t>1$.

For every $s,t\geq 0$, we set
\[
m_{\bf e}(s,t)=\inf_{r\in[s\wedge t,s\vee t]}{\bf e}(r),
\]
and, the pseudo-metric
\[
d_{\bf e}(s,t)={\bf e}(s)+{\bf e}(t)-2m_{\bf e}(s,t).
\]

We then introduce the equivalence relation
$s\sim t$ iff $d_{\bf e}(s,t)=0$ (or equivalently iff ${\bf e}(s)={\bf e}(t)=m_{\bf e}(s,t)$). Let
$T_{\bf e}$ be the quotient space
\begin{equation}\label{def_equiv_tree}
T_{\bf e}=[0,\infty)/ \sim.
\end{equation}
 For any fixed realization $e$ of the excursion $\bf e$, the pseudo-metric $d_{ e}$ becomes a metric in the quotient space $T_{ e}$ and the metric space $(T_{ e},d_{ e})$, is a \emph{Real Tree}.\footnote{In the sense that $(i)$ For $x,y\in T_{e}$, there exists a unique isometric embedding $f_{x,y}:[0,d_{e}(x,y)]\to T_e$ with $f_{x,y}(0)=a$ and $f_{x,y}(d_{e}(x,y))=b$ and $(ii)$ If $q$ is a continuous injection $q:[0,1]\to T_e$ with $q(0)=x$, $q(1)=y$ we have that $q([0,1])=f_{x,y}([0,d_e(x,y)])$.}

\begin{definition}
\label{CRTdef}
The random real tree $T_{\bf e}$ is called the Continuum Random Tree (CRT) and will be  denoted $(\mathfrak{T},d_{\mathfrak{T}})$. We write $\Xi$ to denote its law.
\end{definition}

The CRT was initially defined by Aldous \cite{Al1} with a different formalism,
but the preceding definition corresponds to Corollary 22 in \cite{Al3}, up to an unimportant
scaling factor $2$.

We can define a natural volume measure on $\mathfrak{T}$ by projecting the Lebesgue measure on $[0,1]$, i.e., for any open $A\subseteq \mathfrak{T}$, we set
\[
\lambda_{\mathfrak{T}}(A)=\text{Leb}\{t\in [0,1], [t]\in A\},\label{eq:defoflambda}
\]
where $[t]$ denotes the equivalence class of $t$ with respect to the relation defined at~\eqref{def_equiv_tree}.

 Let $({\bold e}^h(t))_{t\geq0}$ be a Brownian excursion conditioned on reaching level $h$, i.e., $\max_{t\geq0} {\bold e}^h(t)\geq h$. More precisely, let $(W_t)_{t\geq 0}$ be a standard Brownian motion and $\tau_h=\inf\{t\geq0:W_t=h\}$. Let $I=[a,b]$ be the interval $[\sup\{t\leq\tau_h:W_{t}=0\},\inf\{t\geq\tau_h:W_{t}=0\}]$. Now, let 
 ${\bold e}^h(t)=W(t+a)$, $t\in[0,b-a]$. For $t\geq b-a$, we set ${\bold e}^h(t)=0$.
  \begin{definition}
\label{CRThdef}
The random real tree $T_{{\bf e}^h}$ is called the continuum random tree conditioned on height (written $h$-CRT) and will be denoted $(\mathfrak{T}^h,d_{\mathfrak{T}^h})$. We write $\Xi^h$ to denote its law.

We associated to it the natural volume measure $\lambda_{\frak{T}^h}$ by setting for any open $A\subseteq \mathfrak{T}^h$
\[
\lambda_{\mathfrak{T}^h}(A)=\text{Leb}\{t\in [0,\abs{\mathfrak{T}^h}], [t]\in A\},
\]
where $\abs{\mathfrak{T}^h}=\sup\{t\geq 0, {\bf e}^h(t)>0\}.$
\end{definition}

\subsection{The integrated super-Brownian excursion (ISE) and its counterpart conditioned on height}\label{sect:ISE}

We may consider the $\R^d$-valued Gaussian process
$(\phi_{\mathfrak{T}}(\sigma),\sigma\in \mathfrak{T})$ whose distribution is characterized by
\begin{align*}
&\mathbb{E}[\phi_{\mathfrak{T}}(\sigma)]=0\;,\\
&{\rm cov}(\phi_{\mathfrak{T}}(\sigma),\phi_{\mathfrak{T}}(\sigma'))=d_{\mathfrak{T}}(\text{root},\sigma\wedge \sigma')\,{\rm Id}\;,
\end{align*}
where ${\rm Id}$ denotes the $d$-dimensional identity matrix. This process has a continuous modification (see (8) in~\cite{LG} for details). 
If we replace $\frak T$ by $\frak T^h$ in the definition above, we get an $\mathbb{R}^d$-valued Gaussian Process $(\phi_{\frak{T}^h}(\sigma))_{\sigma\in\frak{T}^h}$.

Given $\mathfrak{T}$, we denote $Q_{\mathfrak{T}}$ the law of $(\mathfrak{T},(\phi_{\mathfrak{T}}(\sigma),\sigma\in \mathfrak{T}))$ and we denote the joint annealed law by
\[
M=\int \Xi(d\mathfrak{T})\,Q_{\mathfrak{T}}.
\]

\begin{definition}
The random probability measure $\lambda_{\phi_{\mathfrak{T}}(\mathfrak{T})}$ on $\R^d$ defined under $M$ by $\lambda_{\phi_{\mathfrak{T}}(\mathfrak{T})}:=\lambda_{\mathfrak{T}} \circ \phi_{\mathfrak{T}}^{-1}$ is called $d$-dimensional ISE (for Integrated Super-Brownian Excursion).
\end{definition}

 The random measure
ISE was first discussed by Aldous~\cite{Al4}.
For $d\geq4$, the topological support of ISE is the range
of $\phi_{\frak{T}}$. Moreover, the measure  $\lambda_{\phi_{\mathfrak{T}}(\mathfrak{T})}$ should be interpreted as the
uniform measure on the range $\phi_{\frak{T}}(\mathfrak{T})$ (see \cite{perkins1988space}, \cite{perkins1989hausdorff} and also \cite{delmas1999some}). We will often abuse the terminology and write ISE to mean its topological support.
Using similar definitions, given $\mathfrak{T}^h$, we denote $Q_{\mathfrak{T}^h}$ the law of $(\mathfrak{T}^h,(\phi_{\mathfrak{T}^h}(\sigma),\sigma\in T))$ and the joint annealed law by
\[
M^h=\int \Xi^h(d\mathfrak{T}^h)\,Q_{\mathfrak{T}^h}.
\]

\begin{definition}
The random measure $\lambda_{\phi_{\mathfrak{T}^h}(\mathfrak{T}^h)}$ on $\R^d$ defined under $M^h$ by $\lambda_{\phi_{\mathfrak{T}^h}(\mathfrak{T}^h)}:=\lambda_{\mathfrak{T}^h} \circ \phi_{\mathfrak{T}}^{-1}$ is called $d$-dimensional Super-Brownian motion conditioned on height (written $h$-ISE).
\end{definition}

\begin{remark} Note that the measure $\lambda_{\phi_{\mathfrak{T}^h}}$ is not a probability measure in general, but its support is exactly $\phi_{\mathfrak{T}^h}(\mathfrak{T}^h)$, see \cite{delmas1999some}. 
\end{remark}

\subsection{The Brownian motion on the ISE: $B^{\text{ISE}}$}\label{sssbmise}

In~\cite{Al2}, Aldous gave a set of properties defining uniquely a process which corresponds to our intuition of a Brownian motion on $(T,d_{T},\nu)$ where $(T,d_T)$ is a real tree and $\nu$ will correspond to an invariant measure of said Brownian motion. 

The existence of a such process $B^{\text{CRT}}$ on the CRT was  proved by Krebs~\cite{krebs}. After, this was generalized using techniques of resistance forms (see~\cite{Kigami_Harm} for an introduction on resistance forms). More specifically, it was proved in Section 6 of~\cite{Croydon_arc} that
\begin{proposition}\label{prop_def_process}
Let $(T,d_{T})$ be a compact real tree, $\nu$ be a finite Borel measure on $\T$ that satisfies $\nu(A)>0$ for every non-empty open set $A\subseteq T$, and $(\mathcal{E}_{T},\mathcal{F}_{T})$ be the resistance form associated with $(T,d_{T})$. Then $(\frac 12 \mathcal{E}_{T},\mathcal{F}_{T})$ is a local, regular Dirichlet form on $L^2(T, \nu)$, and the corresponding Markov process $B^{T,\nu}$ is the Brownian motion on $(T,d_{T},\nu)$.
\end{proposition}

 It was shown in~\cite{Croydon_arc} that for $d\geq 8$ the CRT and the ISE are isometric which allows us to build the $B^{\text{ISE}}$  by embedding the $B^{\text{CRT}}$. This is summed up in the following proposition. Let $d_{\phi_{\frak T}(\frak T)}$ be a distance defined in $\phi_{\frak T}(\frak T)$ as $d_{\phi_{\frak T}(\frak{T})}(x,y):=d_{\frak{T}}(\phi^{-1}_{\frak{T}}(x),\phi^{-1}_{\frak{T}}(y))$ for all $x,y\in\phi_{\frak{T}}(\frak{T})$.

\begin{proposition}\label{propdef_BISE} For $\Xi$-a.e.~$\mathfrak{T}$, the Brownian motion $B^{\text{CRT}}$ on $(\mathfrak{T},d_{\mathfrak{T}},\lambda_{\mathfrak{T}})$ exists. Furthermore if $d\geq 8$, for $M$-a.e.~ $(\mathfrak{T},\phi_{\frak T})$, the Brownian motion $B^{\text{ISE}}$ on $(\phi_{\mathfrak{T}}(\mathfrak{T}),d_{\phi_{\mathfrak{T}}(\mathfrak{T})},\lambda_{\phi_{\frak T}(\mathfrak{T})})$ exists and, moreover, $B^{\text{ISE}}=\phi_{\mathfrak{T}}(B^{\text{CRT}})$.
\end{proposition}

It is straightforward to generalize this statement to prove that $\phi_{{\frak T}^h}$ is injective as long as $d\geq 8$.  This allows us to state the following proposition. Let $d_{\phi_{\frak {T}^h}(\frak {T}^h)}$ be a distance in $\phi_{\frak {T}^h}(\frak {T}^h)$ defined as $d_{\phi_{\frak {T}}(\frak{T})}$ but, with $\frak{T}^h$ in place of $\frak{T}$.

\begin{proposition}\label{propdef_BISE2} For $\Xi^h$-a.e.~$\mathfrak{T}^h$, the Brownian motion $B^{h\text{-CRT}}$ on $(\mathfrak{T}^h,d_{\mathfrak{T}^h},\lambda_{\mathfrak{T}^h})$ exists. Furthermore if $d\geq 8$, for a.e.~ $(\mathfrak{T}^h,\phi_{\frak{T}^h})$, the Brownian motion $B^{h\text{-ISE}}$ on $(\phi_{\mathfrak{T}^h}(\mathfrak{T}^h),d_{\phi_{\mathfrak{T}^h}(\mathfrak{T}^h)},\lambda_{\phi_{\mathfrak{T}^h}(\mathfrak{T}^h)})$ exists and, moreover, $B^{h\text{-ISE}}=\phi_{\mathfrak{T}^h}(B^{\text{h-CRT}})$.
\end{proposition}

\subsection{Graph spatial trees and approximations of the ISE and the $B^{\text{ISE}}$}\label{sect_KISE}

We start the section by introducing the concept of graph spatial trees and we will then introduce approximations of the ISE ($h$-ISE) and the $B^{\text{ISE}}$ ($B^{h\text{-ISE}}$).

\subsubsection{Graph spatial trees}\label{sect_graph_tree}

Let us now present a notion introduced by Croydon in~\cite{Croydon_arc}.
\begin{definition}\label{def_gst}
A graph spatial tree $(T,d_T,\phi_T)$ is a triple, where $T$ is a (rooted, ordered)  tree-graph, $d_T$ is a metric on $T$ determined by edge lengths $(l(e))_{e\in E(T)}$ and $\phi_T:\overline{T}\to\mathbb{R}^d$ is a continuous function, where $\overline{T}$ is the real-tree naturally associated with $(T,d_T)$, i.e., $\overline{T}$ is composed of a finite number of line segments with finite edge length $d_T$. 

We will use the term shape to designate the first component $T$ of a graph spatial tree. We say that a graph spatial tree has a non-degenerate shape if all the vertices of $T$ have degree 1 or 3.
\end{definition}

 Given a graph spatial tree $(T,d_T,\phi_T)$, we can assign a probability measure $\lambda_{T}$  defined as the renormalized Lebesgue measure (so that the $\lambda_{T}$-measure of a line segment in $T$ is proportional to its length).

\vspace{0.5cm}

{\it A simple way to construct graph spatial  trees }

\vspace{0.5cm}

Let $(T,d_{T},\phi_T)$ be a triple where the pair $(T,d_T)$ is a rooted real tree and $\phi_T:T\to\R^d$ be a continuous embedding and we consider a sequence $(\sigma_i)_{i\in \N}$ of points of the real tree $T$. There is a simple way to construct a rooted graph spatial tree from $(T,d_T,\phi_T)$ with an order on the edges.

 Firstly, we introduce some necessary notation. For $x,y\in T$, the unique path between $x$ and $y$ is denoted $[x,y]$. We define the branching point of $\sigma,\sigma',\sigma''\in T$ as the unique point $b^T(\sigma,\sigma',\sigma'')\in T$ which is in the triple intersection of the paths $[\sigma,\sigma']$, $[\sigma',\sigma'']$ and $[\sigma'',\sigma]$.

Fix $K\in \N$. We define the reduced subtree $T(\sigma_1,\ldots,\sigma_K)$ to be the graph tree with vertex set
\[
V(T(\sigma_1,\ldots, \sigma_K)):=\{b^{T}(\sigma,\sigma',\sigma''):\sigma,\sigma',\sigma'' \in \{\text{root},\sigma_1,\ldots,\sigma_K\}\},
\]
and graph tree structure induced by the arcs of $T$, so that two elements $\sigma$ and $\sigma'$ of $V(T(\sigma_1,\ldots,\sigma_K))$ are connected by an edge if and only if $\sigma\neq \sigma'$ and also $[\sigma,\sigma']\cap V(T(\sigma_1,\ldots,\sigma_K))=\{\sigma,\sigma'\}$. We set the length of an edge $\{\sigma,\sigma'\}$ to be equal to $d_{T}(\sigma,\sigma')$ and we extend the distance linearly on that edge. Also, we consider the embedding of $T(\sigma_1,\ldots,\sigma_K)$ given by the restriction of $\phi_T$ to $T(\sigma_1,\ldots,\sigma_K)$. This allows us to view $T(\sigma_1,\ldots, \sigma_K)$ as a graph spatial tree.

Next, we will show how to endow $T(\sigma_1,\ldots,\sigma_K)$ with an order on the edges. The leaves $\sigma_1,\ldots, \sigma_K$ are naturally ordered from 1 to $K$ and, when $T(\sigma_1,\ldots,\sigma_K)$ is non-degenerate, then we can order the remaining branching points in the order  you encounter them when moving from the root to $\sigma_1$, then continue the labeling as you move from the root to $\sigma_2$ and so on up to $\sigma_K$. This creates an ordering on the vertices. Then the label/order of an edge is that of its endpoint furthest from the root.

This graph spatial tree will be denoted $(T^{K,(\sigma_i)},d_{T^{K,(\sigma_i)}}, \phi_{T^{K,(\sigma_i)}})$. The associated normalized probability measure is denoted $\lambda_{\phi_{T^{K,(\sigma_i)}}(T^{K,(\sigma_i)})}$. The dependence on $\sigma$ will often be dropped in the notation when the context is clear.

\subsubsection{The $K$-CRT, the $K$-ISE and the $B^{K\text{-ISE}}$ and their counterparts conditioned on height}\label{sect:kskele}

Consider $\mathfrak{T}$ a realization of the CRT and $(V_i)_{i\in \N}$ chosen according to $(\lambda_{\mathfrak{T}})^{\otimes \N}$.
Fix $K\in \N$ . We can use the construction described in Section~\ref{sect_graph_tree} to define a graph spatial tree, which we call $K$-ISE and denote it  $\mathfrak{B}^{(K)}=(\mathfrak{T}^{(K)},d_{\mathfrak{T}^{(K)}},\phi_{\mathfrak{T}^{(K)}})$, where $(\mathfrak{T}^{(K)},d_{\mathfrak{T}^{(K)}})$ is called  $K$-CRT.  Note that its shape $\mathfrak{T}^{(K)}$ comes with an order on the edges.

We recall that this object comes with a probability measure $\lambda_{\phi_{\mathfrak{T}^{(K)}}(\mathfrak{T}^{(K)})}$. For the sake of simplicity we will denote $\lambda_{\phi_{\mathfrak{T}^{(K)}}(\mathfrak{T}^{(K)})}$ as $\lambda^{(\frak{T},K)}$.

 It is also interesting to note that $\mathfrak{T}^{(K)}$ has no point of degree more than 3, indeed, by Theorem 4.6 in~\cite{DuLG}, it is known that $\Xi$-a.s.~for any $x\in \mathfrak{T}$ the set $\mathfrak{T}\setminus \{x\}$ has at most three connected components. This means that the shape $\mathfrak{T}^{(K)}$ is non-degenerate.

Let $B^{K\text{-ISE}}$ be the Brownian motion on the $K$-ISE $\frak{B}^{(K)}$ (endowed with the measure $\lambda^{(\frak{T},K)}\circ \phi_{\mathfrak{T}^{(K)}}^{-1}$) according to the definition of Brownian motion on a real tree, given in Proposition \ref{prop_def_process}. It can be shown (in essence equation (8.3) of~\cite{Croydon_arc}) that

\begin{proposition} \label{prop_approx_bisek}
We have that $B^{K\text{-ISE}}$ converges to $B^{\text{ISE}}$ as $K\to \infty$, in distribution  in the topology of uniform convergence (in compacts subsets of time) in $C(\R_+,\R^d)$ for $M \otimes (\lambda_{\mathfrak{T}})^{\otimes \N}$-a.e.~realization of $(\mathfrak{T},d_{\mathfrak{T}},\phi_{\mathfrak{T}},(V_i)_{i\in \N})$.
\end{proposition}

A similar construction  can be made by choosing, on a realization $\mathfrak{T}^h$ of the $h$-CRT, points  $(V_i)_{i\in \N}$ distributed  according to $(\frac{\lambda_{\mathfrak{T}^h}}{\lambda_{\mathfrak{T}^h}(\mathfrak{T}^h)})^{\otimes \N}$ (meaning they are uniformly distributed on $\mathfrak{T}^h$). The reduced sub-tree spanned by the root and $(V_i)_{i=1,\dots,K}$ will be denoted as $\mathfrak{T}^{h,K}$). The tree $\mathfrak{T}^{h,K}$ comes endowed with a probability measure as in Section \ref{sect_KISE} which will be denoted $\lambda^{(\frak T^h,K)}$. We define $B^{h,K\text{-ISE}}$ the Brownian motion in $(\frak{T}^{h,K},d_{\frak{T}^{h,K}},\lambda^{(\frak{T}^h,K)})$ according to Proposition \ref{prop_def_process}.

\begin{remark}\label{rk:equivalence}
In the related article \cite{CFHP}, a random graph spatial tree $\mathcal{B}_K(W)$ is constructed by uniformly sampling paths from the measure of the historical super Brownian motion. It can be shown that $\mathcal{B}_K(W)$ has the same law as $\mathfrak{B}^{(K)}$. The proof is a straightforward generalization of the argument in the proof of Theorem 2.1 in \cite{le1993class}
\end{remark}

\section{The skeleton of a graph} \label{sect_abstract_theorem}

Let $G=(V(G),E(G))$ be a graph, where $V(G)\subset\mathbb{Z}^d$ containing the origin, $o \in V(G)$. We will have in mind the nearest neighbor case, where $E(G)\subset E(\mathbb{Z}^d)$, or the spread out case, where there exists $L>0$ such that $\|x-y\|>L$ implies $(x,y)\notin E(G)$.
We are now going to introduce the skeleton of $G$ and related notations which are needed to state the four conditions $(G)$, $(R)$, $(V)$ and $(S)$ for the convergence to the $B^{\text{ISE}}$. A more detailed description can be found in~\cite{BCFa}.

One of the key notions we will need in this paper is the notion of cut-point.
\begin{definition}\label{def_cut_point}
We call cut-bond any edge $e\in E(G)$ whose removal disconnects G. The endpoints of a cut-bond are called a cut-points. \end{definition}
We denote $V_{\text{cut}}(G)$ the set of cut-points of $G$, which we assume  to be non-empty.

 Let us now consider a sequence  $(x_i)_{i\in \N}$ of points $V_{\text{cut}}(G)$. Fix $K\in \N$, we construct the graph $G(K)$ in the following manner
\begin{enumerate}
\item the vertices of $G(K)$ are the set of all cut-points that lie on a path from the root to an $x_i$ for $i\leq K$,
\item  two vertices of $G(K)$ are adjacent if there exists a path connecting them which does not use any cut-point.
\end{enumerate}

The new graph $G(K)$ will be rooted at $\text{root}^*$  which is the first cut-point on the path from the origin to $x_1$.

The removal of all cut-bonds in $G$ results in a with graph several connected components. Those connected components are called bubbles and all cut-points  in the same bubble are inter-connected in $G(K)$. This means that $G(K)$ is a graph is composed of complete graphs glued together but single edges. This construction can be visualized in Figure~1.

\begin{definition}\label{def_thin}
We will say that a graph $G(K)$ is tree-like if it does not contain any subgraph that is a complete graph apart from segments and triangles.
\end{definition}

  \begin{figure}\label{fig0}
  \includegraphics[width=0.8\linewidth]{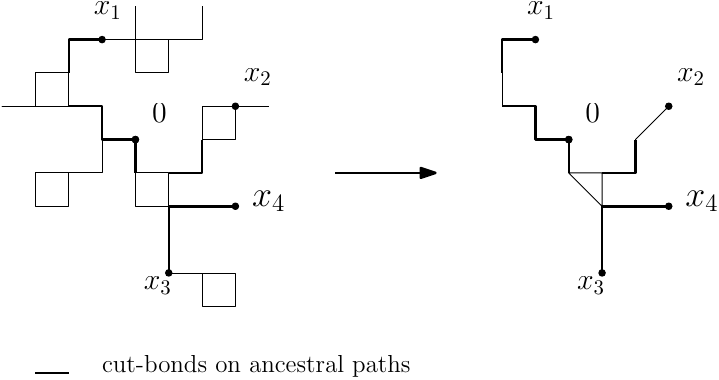}
  \caption{The transformation a graph $G$ is turned into $G(4)$.}
\end{figure}

\begin{remark}\label{skeleton_lat}
In the case of lattice trees, all edges are cut-bonds and all vertices are cut-points. In particular, the graph $G(K)$ is always tree-like.
\end{remark}

\subsection{Approximating a tree-like graph by a graph spatial tree}\label{sec:treelike}

Let us assume that $G(K)$ is tree-like. We are now going to perform a technical operation, that will be helpful to complete our proofs. In essence we are trying to build a graph spatial tree that will approximate $G(K)$ well.

We want to turn the triangles  present in $G(K)$ into stars in order to turn out the tree-like graph into a tree, this procedure will add one point for every triangle present in the graph $G(K)$.

\begin{remark}\label{skeleton_lat2}
In the case of lattice trees, step 1 and step 2 described below are not necessary since $G(K)$ has no actual bubbles. In the end, the graph $\T^{(G,K)}$ constructed will simply be the subgraph of $G$ composed by the union of the ancestral paths of the points $x_1,\ldots, x_K$.
\end{remark}

\vspace{0.5cm}

{\it Step 1: Turning $G(K)$ into a tree $\T^{(G,K)}$}

\vspace{0.5cm}

 For every triangle $(x,y),(y,z),(z,x)\in E(G(K))$, we remove the edges $(x,y),(y,z),(z,x)$ and we introduce a new vertex $v_{x,y,z}$ and new edges $(x,v_{x,y,z})$, $(y,v_{x,y,z})$, $(z,v_{x,y,z})$.   We denote $\T^{(G,K)}$, the tree obtained by this construction.

 We denote $V(\T^{(G,K)})$ the vertices of $\T^{(G,K)}$ and $V^*(\T^{(G,K)})$ the vertices which are not of the form $v_{x,y,z}$ (which are actually the vertices of $G(K)$).

 Similarly, we denote $E(\T^{(G,K)})$ the edges of $\T^{(G,K)}$ and $E^*(\T^{(G,K)})$  the edges which are not of the form $(x,v_{x,y,z}),(y,v_{x,y,z}),(z,v_{x,y,z})$.

Finally, for $x,y\in V(\T^{(G,K)})$, we write $x\sim^*y$ if there exists no $z\in  V^*(\T^{(G,K)})$ which lies on the path from $x$ to $y$. This means that $x$ and $y$ were neighbours before the star-triangle transformation, or equivalently that they are connected by a bubble.

 Since the tree $\T^{(G,K)}$ is rooted (at $\text{root}^*$) it comes with a natural notion of ancestry. For $x\in \T^{(G,K)}$, we denote $\overrightarrow{\T^{(G,K)}_{x}}$, the set of points of $\T^{(G,K)}$ which are descendants of $x$, including $x$.
 
 \begin{remark} In the case where $G$ is a (lattice) tree, we have that  $V^*(\T^{(G,K)})= V(\T^{(G,K)})$, $E^*(\T^{(G,K)})= E(\T^{(G,K)})$ and $\text{root}^*=\text{root}$.\end{remark}

\vspace{0.5cm}

{\it Step 2: Turning $\T^{(G,K)}$ into a real tree by adding a metric}

\vspace{0.5cm}

The tree $\T^{(G,K)}$ comes with a natural metric by setting
\begin{enumerate}
\item for $(x,y)\in E^*(\T^{(G,K)})$, we set $d_{\T^{(G,K)}}(x,y)=d_{G}(x,y)$, where $d_G$ is the graph distance,
\item for any triple of edges $(x,v_{x,y,z}),(y,v_{x,y,z}),(z,v_{x,y,z})$, where $x$ is the ancestor of $y$ and $z$, we set \[d_{\T^{(G,K)}}(x,v_{x,y,z})=\frac{d_G(x,y)+d_G(x,z)-d_G(z,y)}{2},\] \[d_{\T^{(G,K)}}(y,v_{x,y,z})=\frac{d_G(x,y)+d_G(y,z)-d_G(x,z)}{2}\] and \[d_{\T^{(G,K)}}(z,v_{x,y,z})=\frac{d_G(x,z)+d_G(y,z)-d_G(x,y)}{2}.\]
Note that this assignment of distances keeps consistency in the sense that \[d_G(x,y)=d_{\T^{(G,K)}}(x,v_{x,y,z})+d_{\T^{(G,K)}}(y,v_{x,y,z}),\] \[d_G(x,z)=d_{\T^{(G,K)}}(x,v_{x,y,z})+d_{\T^{(G,K)}}(z,v_{x,y,z})\] and \[d_G(y,z)=d_{\T^{(G,K)}}(y,v_{x,y,z})+d_{\T^{(G,K)}}(z,v_{x,y,z}).\]
\item the distance grows linearly along an edge.
\end{enumerate}

Our choice for the distances in the second part is arbitrary but it will not have an significant impact on our proof. It can be noted that this distance conserves the distance from $\text{root}^*$ to any point in $V^*(\T^{(G,K)})$.

\vspace{0.5cm}

{\it Step 3: Assigning a spatial location to the points in $\T^{(G,K)}$}

\vspace{0.5cm}

Finally we want to view our tree as a spatial tree embedded in $\R^d$, i.e.~we want to find an embedding of the edges into $\R^d$.

 Any vertex of $V^*(\T^{(G,K)})$ is assigned its original location in $G$. Moreover the vertices $v_{x,y,z}$ are mapped to the barycenter of  $x$, $y$ and $z$. We write $\phi^{G(K)}$ this map.

  If $(x,y)\in E(\T^{(G,K)})$, then the point $z$ which is at a $d_{\T^{(G,K)}}$-distance $\alpha d_{\T^{(G,K)}}(x,y)$ along the edge $(x,y)$ is mapped to the point which is at distance $\alpha d_{\Z^d}(\phi^{G(K)}(x),\phi^{G(K)}(y))$ along the $\R^d$-geodesic between $\phi^{G(K)}(x)$ and $\phi^{G(K)}(y)$. This extends  $\phi^{G(K)}$ to a map from $\T^{(G,K)}$ to $\R^d$.

  In particular the notation $\phi^{G(K)}(e)$, for $e\in E(\T^{(G,K)})$, corresponds to a segment of $\R^d$.

\subsection{Adding a measure associated to the volume of the graph}\label{sect_mu}

We are going to add a measure to our graph $\T^{(G,K)}$.

For any $x\in G$, let $\pi^{(G,K)}(x)$ be the unique $v\in V^*(\T^{(G,K)})$ separating $x$ from the origin and such that for any $v'\in V^*(\T^{(G,K)})$ with $v'$ separating $x$ from the origin and $v'\neq v$ we have that $v' \prec v$. That is, when going from $\text{root}^*$ to $x$, the point $\pi^{(G,K)}(x)$ is the last cut-point crossed before reaching $x$. In the case where $x$ is not separated from the origin by a cut-bond, i.e.~$x$ is in the bubble of the origin, then we set $\pi^{(G,K)}(x)=\text{root}^*$ by convention.

Now for $x\in V^*(\T^{(G,K)})$, let $v_{\T^{(G,K)}}(x):=\#\{(y,z)\in E(G): \pi^{(G,K)}(y)=x \text{ and } y\neq x\}$ and use this to define a measure on $V^*(\T^{(G,K)})$.
\begin{equation}\label{eq:defofmu}
\mu^{(G,K)}:=\sum_{x\in V^*(\T^{(G,K)})} v_{\T^{(G,K)}}(x)\delta_x.
\end{equation}

\subsection{Another way of viewing $\T^{(G,K)}$ as graph spatial tree}\label{sect_mathfrak}

For our future purpose it will be convenient to be able to introduce a reduced version of $\T^{(G,K)}$. This distinction will be important for the Definition~\ref{def_condG}.

It will be a graph spatial tree  denoted $\mathfrak{B}^{(G,K)}=(T^{(G,K)},d_{T^{(G,K)}},\phi_{T^{(G,K)}})$ which is obtained by the procedure described in Section~\ref{sect_graph_tree}. In the notations of that section this spatial graph is $((\T^{(G,K)})^{K,(x_i)},d_{(\T^{(G,K)})^{K,(x_i)}}, \phi_{(\T^{(G,K)})^{K,(x_i)}})$.

In words, to build $T^{(G,K)}$, one considers the subgraph of $\T^{(G,K)}$ connecting $\text{root}^*$, $x_1,\ldots,x_K$ then every vertex of degree 2 is erased (collapsing the 2 adjacent edges into one).  In particular, the vertices of $T^{(G,K)}$ are $\text{root}^*$, $x_1,\ldots,x_K$ and the corresponding branching points.

The graph spatial tree $\mathfrak{B}^{(G,K)}$ is called the $K$-skeleton of $G$.

\begin{remark}\label{rem_not_depend} It is important to note that the distance and the embedding we assign to $T^{(G,K)}$ coincide with those assigned to $\T^{(G,K)}$. This will allow us to use, e.g., $d_{\T^{(G,K)}}$ to signify $d_{T^{(G,K)}}$.\end{remark}

\begin{remark}\label{skeleton_lat3}
In the case of lattice trees, $T^{(G,K)}$ is composed of the vertices $\text{root}$, $x_1, \ldots, x_k$ (as well as their branching points) and two of those vertices are adjacent in $T^{(G,K)}$ if, and only if, they can be connected in $\T^{(G,K)}$ without using a vertex of $T^{(G,K)}$ .
\end{remark}

T
\section{Conditions $(S)$, $(G)$, $(V)$ and $(R)$} \label{sect_conditions}

Our goal in this section is to describe the four conditions defined in~\cite{BCFa} which imply convergence of the simple random walks on certains critical graphs $G_n$ towards the $B^{h\text{-ISE}}$. For this, we will consider a sequence of random graphs $(G_n)_{n\in \N}$.

Fix $K\in \N$. If the graph $G_n(K)$ constructed from $(G_n,(V_i^n)_{i\in \N})_{n\in \N}$ is tree-like, then the construction of the skeleton of the previous section can be carried out. In order to lighten the notations, we will write $\T^{(n,K)}$, $T^{(n,K)}$, $V^*(\T^{(n,K)})$, $\phi^{(n,K)}$ and $\pi^{(n,K)}$ for  $\T^{(G_n,K)}$, $T^{(G_n,K)}$, $V^*(\T^{(G_n,K)})$, $\phi^{G_n(K)}$,  and $\pi^{(G_n,K)}$ and we also introduce the rescaled quantities $d^{(n,K)}(\cdot,\cdot)$ and $\mu^{(n,K)}$ for $n^{-1}d_{\T^{(G_n,K)}}(\cdot,\cdot)$ and $\abs{E(G_n)}^{-1}\mu^{(G_n,K)}$. All those quantities were defined in the previous section.

We recall that $d^{(n,K)}(\cdot,\cdot)$ (resp.~$\phi^{(n,K)}$) is a distance on (resp.~embedding of) $T^{(n,K)}$ because of Remark~\ref{rem_not_depend}.


\subsection{Condition $(S)$}\label{sect_asympT_thin}

For any $x\in V^*(\T^{(n,K)})$, we call $K$-sausage of $x$ the set $\{y\in G_n,\ \pi^{(n,K)}(y)=x\}$.

 Note that a sausage is typically much larger than the corresponding bubble because it also contains bubbles of $G_n$ which are not in $\T^{(n,K)}$. We introduce
\begin{equation}\label{eq:defofdelta}	
\Delta^{(n,K)}_{\Z^d}:=\max_{x\in V^*(\T^{(n,K)})} \text{Diam}_{\Z^d}(\{y\in V(G_n),\ \pi^{(n,K)}(y)=x\}),
\end{equation}
where $\text{Diam}_{\Z^d}(A):=\max\{d_{\Z^d}(x,y):x,y\in A\}$, for any $A\subset \Z^d$.
We also introduce
\begin{equation}\label{eq:defofdeltaintr}
\Delta^{(n,K)}_{G_n}:=\max_{x\in V^*(\T^{(n,K)})} \text{Diam}_{G_n}(\{y\in V(G_n),\ \pi^{(n,K)}(y)=x\}),
\end{equation}
where $\text{Diam}_{G_n}(A):=\min\{d_{G_n}(x,y):x,y\in A\}$ for any $A\subset \Z^d$, and $d_{G_n}$ is the graph distance in $G_n$.
In our context, we want to extend the definition of tree-like graphs (see Definition~\ref{def_thin}).
\begin{definition}\label{def_athin}
We say that a sequence of random augmented graphs $(G_n,(V_i^n)_{i\in \N})_{n\in \N}$ verifies condition $(S)$ if
\begin{enumerate}
\item for all $K\in \N$, we have
\[
\lim_{n\to \infty} {\mathbb{P}}[G^{(n,K)}\text{ is tree-like}]=1.
\]
\item for all $\epsilon>0$, we have
\begin{enumerate}
\item \[
 \lim_{K\to \infty} \sup_{n\in \N} {\mathbb P}[n^{-1/2}\Delta^{(n,K)}_{\Z^d} >\epsilon]=0\]
and
\item
\[
\lim_{K\to \infty} \sup_{n\in \N} {\mathbb P}[n^{-1}\Delta^{(n,K)}_{G_n} >\epsilon]=0.
\]
\end{enumerate}
\end{enumerate}
\end{definition}

\begin{remark}\label{rem_abuse_tnk} If a sequence $(G_n,(V_i^n)_{i\in \N})_{n\in \N}$ is asymptotically tree-like (meaning that it satisfy the first display of Definition \ref{def_athin}), then the notation $\T^{(n,K)}$ and $T^{(n,K)}$ make sense with probability going to $1$ since these objects can be constructed with the methods of the previous section. The conditions which will involve $\T^{(n,K)}$ and $T^{(n,K)}$ (condition $(G)$ of definition~\ref{def_condG} and condition $(V)$ of definition~\ref{def_condV}) are all asymptotical in $n$. Hence, they are not affected by the fact that $\T^{(n,K)}$ is not defined on an event of small probability. We will thus allow ourselves a slight abuse of notation in the statement of these conditions. \end{remark}

\subsection{Condition $(G)$: asymptotic shape of the graph}\label{sect_condG}

Let us define a distance $D$ on graph spatial trees (defined in Section~\ref{sect_graph_tree}). Here, we follow Section 7 of~\cite{Croydon_arc}.

For $(T,d,\phi)$ a graph spatial tree with an order on the edges, write  $\abs{e_1},\ldots,\abs{e_l}$ for the lengths of the edges.

Take two such graph spatial trees $\mathscr{T}=(T,d,\psi)$ and $\mathscr{T}'=(T',d',\psi ')$. If $T\neq T'$ (i.e., if there is no root preserving, order preserving, graph-isomorphism between $T$ and $T'$) then  we set $d_1(\mathscr{T},\mathscr{T}')=\infty$ and otherwise we set
\begin{equation}\label{eq:defofd1}
d_1(T,T'):=\sup_{i} \abs{\abs{e_i}-\abs{e_i'}}.
\end{equation}

Now if $T=T'$, we have a homeomorphism $\Upsilon_{T,T'}:\overline{T}\to \overline{T'}$ such that  if $x\in \overline{T}$ is at a distance $\alpha \abs{e}$ along the edge $e$, it is mapped to the point $x'\in \overline{T'}$ which is at distance $\alpha \abs{e'}$ along the corresponding edge $e'$. We then set
\[
d_2(T,T'):=\sup_{x\in \overline{T}} d_{\R^d}(\psi(x), \psi'(\Upsilon_{T,T'}(x))).
\]
This yields a metric
\begin{equation}\label{eq:defofD}
D((T,d,\psi),(T',d',\psi ')):=(d_1(T,T')+d_2(T,T'))\wedge 1
\end{equation}
 on graph spatial trees with ordered edges. This distance allows us to define our first condition (relevant definitions can be found at Definition~\ref{def_athin}, Remark~\ref{rem_abuse_tnk} and Section~\ref{sect_KISE}).

\begin{definition}\label{def_condGh} Condition $(G)^h_{\sigma_d,\sigma_{\phi}}$: We say that a sequence of random augmented graphs $(G_n,(V_i^n)_{i\in \N})_{n\in \N}$ satisfies condition $(G)_{\sigma_d,\sigma_{\phi}}^h$ if \begin{enumerate}
\item for all $K\in \N$, we have
\[
\lim_{n\to \infty} {\mathbb P}[G^{(n,K)}\text{ is tree-like}]=1,
\]
\item there exists $\sigma_d, \sigma_{\phi}>0$ such that for all $K\in \N$, the sequence of graph spatial trees $((T^{(n,K)},n^{-1}d_{G_n},n^{-1/2}\phi^{(n,K)}))_{n\in \N}$ converges weakly to $(\mathfrak{T}^{h,K},\sigma_{d} d_{\mathfrak{T}^{h,K}},\sigma_{\phi} \sqrt{\sigma_{d} } \phi_{\mathfrak{T}^{h,K}})$ in the topology induced by $D$.
\end{enumerate}
\end{definition}

\begin{remark} It is very important to stress that in condition $(G)^h$, the topology induced by $D$ imposes a condition on the convergence of the length of only a finite number of edges (the $2K-1$ edges of the asymptotically non-degenerate shape). This is where the distinction between $T^{(n,K)}$ and $\T^{(n,K)}$ makes a big difference. \end{remark}

Let us also introduce the strengthened version of condition $(G)^h$.
\begin{definition}\label{def_condGh2} Condition $(G)^{h,+}$: We say that a sequence of random augmented graphs $(G_n,(V_i^n)_{i\in \N})_{n\in \N}$ satisfies condition $(G)_{\sigma_d,\sigma_{\phi},\nu}^{h,+}$ if
\begin{enumerate}
\item for all $K\in \N$, we have
\[
\lim_{n\to \infty} {\mathbb P}[G^{(n,K)}\text{ is tree-like}]=1,
\]
\item there exists $\sigma_d, \sigma_{\phi},\nu>0$ such that for all $K\in \N$, the sequence of graph spatial trees $((T^{(n,K)},n^{-1}d_{G_n}(\cdot,\cdot),n^{-1/2}\phi^{(n,K)}, \frac{\abs{E(G_n)}}{n^2}))_{n\in \N}$ converges weakly to $(\mathfrak{T}^{h,K},\sigma_{d} d_{\mathfrak{T}^{h,K}},\sigma_{\phi} \sqrt{\sigma_{d} } \phi_{\mathfrak{T}^{h,K}},\nu\lambda_{\frak T^h}(\frak{T}^h))$ with the topology induced by $D$ in the first three coordinates and the usual topology in $\R$ in the last coordinate.
\end{enumerate}
\end{definition}

\subsection{Condition $(V)$: distribution of the volume on the graph}\label{sect_condV}

Define $\lambda^{(n,K)}$ the Lebesgue measure on the graph spatial tree, $(\T^{(n,K)},d^{(n,K)})$, normalized to have total mass $1$. Recall that $\overrightarrow{\T^{(n,K)}_{x}}$ are the descendants of $x$ (including $x$ itself) in $\T^{(n,K)}$ and $\mu^{(n,K)}$ was defined in Section~\ref{sect_mu}. Let us introduce the following definition.

\begin{definition} \label{def_condV}
Condition $(V)$:
For each $\epsilon>0$
\[
\lim_{K\to\infty}\limsup_{n\in\N}{\mathbb P}\left[ \sup_{x\in\T^{(n,K)}} \abs{\lambda^{(n,K)}(\overrightarrow{\T_x^{(n,K)}})-\mu^{(n,K)}(\overrightarrow{\T^{(n,K)}_x})}\geq\epsilon\right]=0.
\]
\end{definition}

\subsubsection{Condition $(R)$: the linearity of the resistance}

Let $\reff^{G_n}$ denote the effective resistance in $G_n$. That is, we let each edge of $G_n$ have unit resistance and, for a pair of vertices $x,y$, $\reff^{G_n}(x,y)$ is the electrical resistance between $x$ and $y$ in the electrical network just described.
\begin{definition} Condition $(R)$: We say that a sequence of random augmented graphs $(G_n,(V_i^n)_{i\in \N})_{n\in N}$  satisfies condition $(R)_{\rho}$ if there exists $\rho>0$ such that for all $\epsilon>0$ and for all $i\in \N$
\[
\lim_{n\to \infty} \mathbb{P}\Bigl[ \abs{\frac{\reff^{G_n}(0,V_i^n)}{d_{G_n}(0,V_i^n)}-\rho}>\epsilon\Bigr]=0.
\]
\end{definition}

\section{Scaling limit results for simple random walks on critical graphs in high dimensions}\label{sect_thms}

In this section, we will state some of the main results of this article, which simplify the conditions under which we can deduce convergence towards the Brownian motion in the ISE.

The article~\cite{BCFa} was dedicated to finding conditions under which simple random walks on critical graphs would converge towards the Brownian motion on the ISE (or minor variants of that process). One of the main results obtained there is the following.
\begin{proposition}\label{thm:height}
Under conditions $(G)^h_{\sigma_d,\sigma_\phi}$, $(S)$, $(V)$ and $(R)_{\rho}$ we have that
\[
(n^{-1/2}X^{G_n}_{n\abs{E(G_n)}t})_{t \geq 0}\stackrel{n\to\infty}{\to} (\sqrt{\sigma_d} \sigma_\phi B^{h\text{-ISE}}_{(\rho\sigma_d)^{-1}\lambda_{\frak{T}^h}(\frak{T}^h)t})_{t\geq 0},
\]
where the convergence is annealed and occurs with the topology of uniform convergence over compact sets of time.
\end{proposition}
If, in addition, we also have the control on the cardinality of $G_n$, we can get rid of the dependence on $|E(G_n)|$ on the time scaling. The following corollary is a simple consequence of the theorem above.
\begin{corollary}
 Under conditions $(G)^{h,+}_{\sigma_d,\sigma_\phi,\nu}$, $(S)$, $(V)$ and $(R)_{\rho}$ we have that
\[
(n^{-1/2}X^{G_n}_{n^3t})_{t \geq 0}\stackrel{n\to\infty}{\to} (\sqrt{\sigma_d} \sigma_\phi B^{h\text{-ISE}}_{(\rho\sigma_d\nu)^{-1}t})_{t\geq 0},
\]
where the convergence is annealed and occurs with the topology of uniform convergence over compact sets of time.
\end{corollary}
The conditions $(G)^h_{\sigma_d,\sigma_\phi}$, $(S)$, $(V)$ and $(R)_{\rho}$ are linked to a choice of certain points $(V_n^i)_{n\in \N}$ in our graphs $G_n$. One of the main contributions of this article is to show that a specific choice of such points $(V_n^i)_{n\in \N}$  leads to significant simplifications of those conditions.


\subsection{Simplified scaling limit theorem}\label{sec:simscalimthm}

For a set  $B\subset V(G_n)$,  we denote by $B^*$ the graph with vertices $B$ and with edges $\{[x,y] \in E(G_n),\ x,y \in B\}$. Let us start by defining a specific class of points.
\begin{definition} \label{def_edge_unif}
We say that the points $(V_i^n)_{i\in\N},n\in\N$ are chosen according to uniform edge-volume if, for each $n\in\N$ the sequence $(V_i^n)_{i\in\N}$ is i.i.d.~and
\[
\lim_{n\to \infty}  {\mathbb P}\left[\max_{\substack{B\subset V(G_n) \\ B^* \text{connected}}} \abs{{\mathbb P}[V_1^n \in B\mid G_n] -\frac{\abs{E(B^*)}}{\abs{E(G_n)}} }>\epsilon\right]=0
\]
for all $\epsilon >0$.
\end{definition}

For points chosen in such a way, we do not need to verify condition $(V)$ but we only require the much simpler condition that $\frac{|E(G_n)|}{n^2}$ converges to $\nu$ (this hypothesis is the difference between $(G)^h_{\sigma_d,\sigma_\phi}$ and $(G)^{h,+}_{\sigma_d,\sigma_\phi,\nu}$). Next we present one  of the main theorems of this article:
\begin{theorem}\label{thm:height2}
Assume the points $(V_i^n)_{i\in \N}$ are chosen according to uniform edge-volume.
\begin{enumerate}
\item Under conditions $(G)^h_{\sigma_d,\sigma_\phi}$, $(S)$ and $(R)_{\rho}$ we have that
\[
(n^{-1/2}X^{G_n}_{n\abs{E(G_n)}t})_{t \geq 0}\stackrel{n\to\infty}{\to} (\sqrt{\sigma_d} \sigma_\phi B^{h\text{-ISE}}_{(\rho\sigma_d)^{-1}\lambda_{\frak{T}^h}(\frak{T}^h)t})_{t\geq 0}.
\]
\item Under conditions $(G)^{h,+}_{\sigma_d,\sigma_\phi,\nu}$, $(S)$ and $(R)_{\rho}$ we have that
\[
(n^{-1/2}X^{G_n}_{n^3t})_{t \geq 0}\stackrel{n\to\infty}{\to} (\sqrt{\sigma_d} \sigma_\phi B^{h\text{-ISE}}_{(\nu\rho\sigma_d)^{-1}t})_{t\geq 0},
\]
\end{enumerate}
where the convergences are annealed and occur with the topology of uniform convergence over compact sets of time.
\end{theorem}

The theorem stated just above will be used in Section \ref{sect_lattice_trees} to show that the random walk on critical lattice trees, in high dimensions, converge to the Brownian motion on the ISE.



\subsection{Extension to other conditionings}

The initial article~\cite{BCFa} was applied in~\cite{BCFb} for simple random walks on critical branching random walks conditioned on the volume being large. In such a setting the limiting process is the Brownian motion on the ISE (which is the super Brownian motion of total measure $1$). 
The methods of proof used in the current article are also applicable in that setting.

Let us introduce the condition $(G)$ for graphs $G_n$ conditioned on having volume asymptotically equal to $n^2$ (instead of height larger than $n$).

 \begin{definition}\label{def_condG} Condition $(G)_{\sigma_d,\sigma_{\phi}}$: We say that a sequence of random augmented graphs $(G_n,(V_i^n)_{i\in \N})_{n\in \N}$ satisfies condition $(G)_{\sigma_d,\sigma_{\phi}}$ if \begin{enumerate}
\item for all $K\in \N$, we have
\[
\lim_{n\to \infty} {\mathbb P}[G^{(n,K)}\text{ is tree-like}]=1,
\]
\item there exists $\sigma_d, \sigma_{\phi}>0$ such that for all $K\in \N$, the sequence of graph spatial trees $((T^{(n,K)},n^{-1}d_{G_n}(\cdot,\cdot),n^{-1/2}\phi^{(n,K)}))_{n\in \N}$ converges weakly to $(T^{(K)},\sigma_{d} d_{T^{(K)}},\sigma_{\phi} \sqrt{\sigma_{d} } \phi_{T^{(K)}})$ in the topology induced by $D$.
\end{enumerate}
\end{definition}
Proving condition $(G)$ for graphs conditioned on having a volume exactly equal to $n$ is in general more complicated, and as far as we know, not proved in any models belonging to the lace expansion class (apart from branching random walks). 

In \cite{BCFa} it is proven that:
\begin{proposition}\label{prop:height2}
 Under conditions $(G)_{\sigma_d,\sigma_\phi}$, $(S)$, $(V)$ and $(R)_{\rho}$ we have that
\[
(n^{-1/2}X^{G_n}_{n\abs{E(G_n)}t})_{t \geq 0}\stackrel{n\to\infty}{\to} (\sqrt{\sigma_d} \sigma_\phi B^{\text{ISE}}_{(\rho\sigma_d)^{-1}})_{t\geq 0},
\]
where the convergence is annealed and occurs with the topology of uniform convergence over compact sets of time.
\end{proposition}
As a simple corollary of the proposition above, we get that:
\begin{corollary}
 If, in addition to conditions $(G)_{\sigma_d,\sigma_\phi}$, $(S)$, $(V)$ and $(R)_{\rho}$ we have that there exists $\nu>0$ such that $\frac{|E(G_n)|}{n^2}$ converges to $\nu$ in probability,
\[
(n^{-1/2}X^{G_n}_{n^3t})_{t \geq 0}\stackrel{n\to\infty}{\to} (\sqrt{\sigma_d} \sigma_\phi B^{\text{ISE}}_{(\rho\sigma_d\nu)^{-1}t})_{t\geq 0},
\]
where the convergence is annealed and occurs with the topology of uniform convergence over compact sets of time.
\end{corollary}

We will establish the following results:
\begin{theorem}\label{thm:volume}
Assume the points $(V_i^n)_{i\in \N}$ are chosen according to uniform edge-volume.
\begin{enumerate}
\item Under conditions $(G)_{\sigma_d,\sigma_\phi}$, $(S)$ and $(R)_{\rho}$ we have that
\[
(n^{-1/2}X^{G_n}_{n\abs{E(G_n)}t})_{t \geq 0}\stackrel{n\to\infty}{\to} (\sqrt{\sigma_d} \sigma_\phi B^{\text{ISE}}_{(\rho\sigma_d)^{-1}t})_{t\geq 0}.
\]
\item If, in addition to conditions $(G)_{\sigma_d,\sigma_\phi}$, $(S)$ and $(R)_{\rho}$, there exists $\nu>0$ such that $\frac{|E(G_n)|}{n^2}$ converges to $\nu$ in probability, we have that
\[
(n^{-1/2}X^{G_n}_{n^3t})_{t \geq 0}\stackrel{n\to\infty}{\to} (\sqrt{\sigma_d} \sigma_\phi B^{\text{ISE}}_{(\nu\rho\sigma_d)^{-1}t})_{t\geq 0},
\]
\end{enumerate}
where the convergences are annealed and occur with the topology of uniform convergence over compact sets of time.
\end{theorem}



\section{Construction of an empirical measure to approximate condition $(V)$}\label{sect_gv}

In this section, we aim to show that either of the conditions, $(G)_{\sigma_d,\sigma_{\phi}}$ or $(G)_{\sigma_d,\sigma_{\phi}}^h$, imply condition $(V)$, when the sequence of points are chosen according to uniform edge volume. For the sake of clarity, we will lead the discussion with condition $(G)_{\sigma_d,\sigma_\phi}$ in mind, making the required remarks to deal with condition $(G)^h_{\sigma_d,\sigma_\phi}$ when necessary.

Condition $(V)$ holds exactly for the ISE. That is, recalling from Section \ref{sect:kskele} that $\lambda^{(\frak T,K)}$ is the normalized Lebesgue measure in $(\frak{T}^{(K)},d_{\frak T^{(K)}})$ and letting $\mu^{(\frak T,K)}$ be the measure in $\frak T^{(K)}$ defined as 
\[\mu^{(\frak T,K)}(A)=\lambda_{\frak T}(\{x\in\frak T: \pi^{\mathfrak T^{(K)}}(x)\in A\}).\] Then, we have that
\[
\lim_{K\to\infty}{\mathbb P}\left[ \sup_{x\in\frak T^{(K)}} \abs{\lambda^{(\frak T,K)}(\overrightarrow{\frak T_x^{(K)}})-\mu^{(\frak T,K)}(\overrightarrow{\frak T^{(K)}_x})}\geq\epsilon\right]=0.
\]
For a proof of the display above, see \eqref{eq:exact} in the proof of Proposition \ref{prop_cond_V}.

 Since condition $(G)_{\sigma_d,\sigma_{\phi}}$ tells us that $G_n$ and the ISE are similar for $n$ large, it is natural to believe that $(V)$ should hold. The difficulty is that $(G)_{\sigma_d,\sigma_{\phi}}$ only gives us access to, loosely speaking, finite-dimensional information and and condition $(V)$ is a uniform bound on the tree. The challenge will be to approximate condition $(V)$ using less information. 

For this, we recall that condition $(V)$ is a statement on the measure $\mu^{(n,K)}$ for large $n$ . Our strategy will be to build an \lq\lq empirical measure\rq\rq~$\hat{\mu}^{(n,K,K_1,K_2)}$ that approximates $\mu^{(n,K)}$ and which is measurable with respect to the shape $\T^{(n,K+K_1+K_2)}$ (and we will require no extra information about the volume of $\T^{(n,K)}$). In particular, this measurability condition means that $\hat{\mu}^{(n,K,K_1,K_2)}$ can be understood for large $n$ using condition $(G)_{\sigma_d,\sigma_{\phi}}$. Our condition $(V)$ will then become, up to error terms, a statement on $\hat{\mu}^{(n,K,K_1,K_2)}$. Then, by using condition $(G)_{\sigma_d,\sigma_{\phi}}$, we will turn our discrete problem (linked to $\hat{\mu}^{(n,K,K_1,K_2)}$) into a question on the ISE (linked to a measure $\hat{\mu}^{(\mathfrak{T},K,K_1,K_2)}$). Finally, by the same argument as in the discrete case, this problem on the ISE will  correspond, up to error terms, to proving condition $(V)$ for the ISE (which as mentioned previously is trivially true).

\subsection{Building an empirical measure}

Let us first describe our strategy for building  the empirical measure. We will build the measure $\hat{\mu}^{(n,K,K_1,K_2)}$ on $\mathcal{T}^{(n,K)}$ in two steps.
\begin{enumerate}
\item First, we will find the location of the atoms of the measure. For this, we will first span a high number $K_1$ of  points chosen according to edge volume and look at their projection onto $\mathcal{T}^{(n,K)}$. This number $K_1$ is chosen to ensure that, with high probability, most of the edge-volume of $G_n$ projected onto $\mathcal{T}^{(n,K+K_1)}$ will be located on the newly added branches ($\mathcal{T}^{(n,K+K_1)}\setminus \mathcal{T}^{(n,K)}$).
\item Then, we will find the edge-volume that should be associated to each of the atoms of the measure. This is done by adding a sufficiently high number of points $K_2$ such that we can do a precise empirical approximation of the edge-volume of $G_n$ which will be projected onto each one of the atoms of  
$\hat{\mu}^{(n,K,K_1,K_2)}$. This will essentially follow from the law of large numbers.
\end{enumerate}

\subsubsection{Location of the atoms of $\mu^{(n,K,K_1,K_2)}$}
 
For any $K_1>0$, we will consider the tree $\T^{(n,K+K_1)}$. We introduce the following notations (see Figure \ref{fig1}): 
\begin{enumerate}
\item $y^{(n,K,K_1)}_1, \ldots , y^{(n,K,K_1)}_{l^{(n,K,K_1)}}$ the points of  $\T^{(n,K)}$ which are of degree $3$ in $\T^{(n,K+K_1)}$ but not in $\T^{(n,K)}$. This means that we consider the branching vertices that were added to $\T^{(n,K)}$ in order to be able to connect the $K_1$ points added in $\T^{(n,K+K_1)}$. We point out that $l^{(n,K,K_1)}\leq K_1$.
\item for any $i\leq l^{(n,K,K_1)}$, we denote $s^{(n,K,K_1)}_i$ to be the set of points of $\T^{(n,K+K_1)}$ which are separated from $0$ by $y^{(n,K,K_1)}_i$ (including $y^{(n,K,K_1)}_i$). It is clear  that $s^{(n,K,K_1)}_i$ forms a tree which is connected. 
\item $S^{(n,K,K_1)}$ is formed by the union of all $s^{(n,K,K_1)}_i$ for $i\leq l^{(n,K,K_1)}$.  This is hence a forest with at most $K_1$ connected components. 
\end{enumerate}

\begin{remark}\label{rem_prop_def}
 We can notice that  $\{\pi^{(n,K)}(z),\ z\in S^{(n,K,K_1)}\}=\{y^{(n,K,K_1)}_i,\ i \leq l^{(n,K,K_1)}\}$.
 \end{remark}

\begin{remark}\label{rmk:defofS}
A similar construction can be made in the limiting setting, i.e.~on the continuum random tree $\mathfrak{T}$ (and $\mathfrak{T}^h$). This will lead to analogous notations, $S^{(K,K_1)}$, $(y_i^{(K,K_1)})_{i\leq l^{(K,K_1)}}$ and $(s_i^{(K,K_1)})_{i\leq l^{(K,K_1)}}$.
\end{remark}

Those definitions are illustrated in Figure~2. 

  \begin{figure}\label{fig1}
  \includegraphics[width=0.8\linewidth]{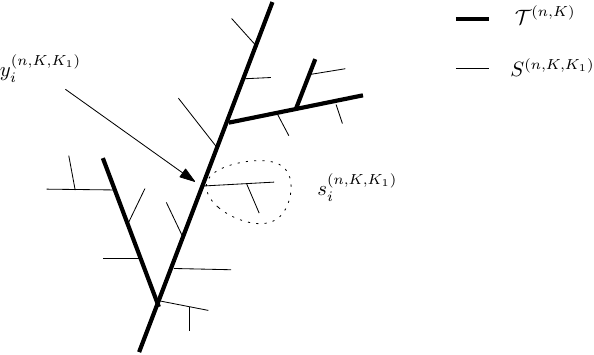}
  \caption{Notations associated to the construction of the atoms of the measure $\hat{\mu}^{(n,K,K_1,K_2)}$}
\end{figure}

We have already defined the projection of vertices at the beginning of the Section~\ref{sect_conditions}.  We will extend this definition to edges in the following way: let $e=[x,y]$ be an edge of $E(G_n)$ and $K'$ any positive integer then 
\begin{itemize}
\item if $\pi^{(n,K')}(x)=\pi^{(n,K')}(y)$ then we set $\pi^{(n,K')}(e)=\pi^{(n,K')}(x)$,
\item if $\pi^{(n,K')}(x)\neq \pi^{(n,K')}(y)$ then if $\pi^{(n,K')}(x)$ is the ancestor of $\pi^{(n,K')}(y)$ in $\mathcal{T}^{(n,K')}$ we set $\pi^{(n,K')}(e)=\pi^{(n,K')}(x)$ and in the other case we set $\pi^{(n,K')}(e)=\pi^{(n,K')}(y)$.
\end{itemize}

This means that $\pi^{(n,K')}(e)$ is the cut-point closest to $e$ such that from either endpoint of $e$ we have to go through $\pi^{(n,K')}(e)$ to go to the origin.

Let us start by a technical result.
\begin{lemma}\label{lem_step0}
Assume that either of Conditions $(G)_{\sigma_d,\sigma_{\phi}}$ or $(G)_{\sigma_d,\sigma_{\phi}}^h$ are verified with points chosen according to uniform edge-volume. Fix $K<\infty$, then for any $M<\infty$ there exists $C_1(M,K)<\infty$ such that for any $K_1\geq C_1$,
\begin{align*}
& \lim_{n \to \infty}{\mathbb P}[\text{ for all } i\leq M^{1/2},\ \pi^{(n,K+K_1)}(V^n_{K+K_1+i})\in S^{(n,K,K_1)}] \\
\geq & \bigl(1-\frac 1M\bigr)\exp\bigl(-\frac 2{M^{1/2}}\bigr).
\end{align*}
\end{lemma}

\begin{proof}
Let $\frak{T}$ denote the normalized CRT, i.e., the one conditioned on having unit volume. Let us introduce the event $A(K,K_1,L)=\{|S^{(K,K_1)}|_{d_{\mathfrak{T}}} \geq L  |\mathfrak{T}^{(K)}|_{d_{\mathfrak{T}}}\}$, where $|\cdot|_{d_{\mathfrak{T}}}$ denotes the length according to $d_{\mathfrak{T}}$. We will start by showing that for any $L<\infty$ and $\eta>0$, there exists $K^\ast(\eta,L)$ such that for all $K_1>K^\ast(\eta,L)$,
\begin{equation}\label{a_crt1}
{\mathbb P}[A(K,K_1,L)] >1-\eta.
\end{equation}

It is immediate from the line-breaking construction of $\frak{T}$ (see Section~\ref{sect_line_break}) that the total length of $\frak{T}^{(K_1)}$ diverges with $K_1$, almost surely. 
 Also, for fixed $K$, the length of $\frak{T}^{(K)}$ is almost surely finite.
 Therefore, since $\abs{S^{(K,K_1)}}_{d_{\frak{T}}}=\abs{{\frak T}^{(K_1)}}_{d_{\frak{T}}}-\abs{{\frak T}^{(K)}}_{d_{\frak{T}}}$, it follows that $\abs{S^{(K,K_1)}}_{d_{\frak{T}}}$ diverges with $K_1$ (for fixed $K$). Display \eqref{a_crt1} follows immediately.

Using the line-breaking construction again, we see that for any $i\geq 1$ the point $\pi^{\mathfrak{T}^{(K+K_1)}}(V_{K+K_1+i})$ is uniformly distributed on $\mathfrak{T}^{(K+K_1)}$ which by \eqref{a_crt1} implies that,
\begin{equation}\label{a_ind0}
{\mathbb P}[\pi^{\mathfrak{T}^{(K+K_1)}}(V_{K+K_1+i}) \in S^{(K,K_1)}\mid A(K,K_1,L)]\geq\frac{L}{L+1}.
\end{equation}

Now, for any $j\geq 1$, let us denote the event 
\[
A_j(K,K_1,L):=\bigcap_{1\leq i\leq j} \{\pi^{\mathfrak{T}^{(K+K_1)}}(V_{K+K_1+j}) \in S^{(K,K_1)}\}\cap A(K,K_1,L),
\]
which we extend by setting $A_0(K,K_1,M):=A(K,K_1,M)$.

 Since $(V_l)_{l\geq 0}$ is an i.i.d.~family, we know that $\{\pi^{\mathfrak{T}^{(K+K_1)}}(V_{K+K_1+j+1}) \in S^{(K,K_1)}\}$ depends on $(V_l)_{l\leq K+K_1}$ (to define $\pi^{\mathfrak{T}^{(K+K_1)}}$ and $S^{(K,K_1)}$) and on $V_{K+K_1+j+1}$, but is independent of $(V_l)_{K+K_1+j\geq l> K+K_1}$. On the event $A_j(K,K_1,L)$, we know that we have $|S^{(K,K_1+j)}|_{d_{\mathfrak{T}}} \geq L  |\mathfrak{T}^{(K)}|_{d_{\mathfrak{T}}}$ and this means that the line-breaking construction implies 
\begin{equation}\label{a_ind1}
{\mathbb P}[\pi^{\mathfrak{T}^{(K+K_1)}}(V_{K+K_1+j+1}) \in S^{(K,K_1)}\mid A_j(K,K_1,L)]
\geq  \frac{L}{L+1}.
\end{equation}

Hence, \eqref{a_ind0} and~\eqref{a_ind1} imply by induction that 
\[
{\mathbb P}[\text{ for all } i\leq j,\ \pi^{\mathfrak{T}^{(K+K_1)}}(V_{K+K_1+i}) \in S^{(K,K_1)}\mid A(K,K_1,L)]\geq \Bigl(\frac{L}{L+1}\Bigr)^j.
\]
Applying for the previous equation for $j=M^{1/2}$, we obtain
\begin{align*}
&{\mathbb P}[\text{ for all } i\leq M^{1/2},\ \pi^{\mathfrak{T}^{(K+K_1)}}(V_{K+K_1+i}) \in S^{(K,K_1)}\mid A(K,K_1,L)] \\
\geq & \bigl(1-\frac 1{L+1}\bigr)^{M^{1/2}},
\end{align*}
and recalling~\eqref{a_crt1} we obtain
\[
{\mathbb{P}}[\text{ for all } i\leq M^{1/2},\ \pi^{\mathfrak{T}^{(K+K_1)}}(V_{K+K_1+i})\in S^{(K,K_1)}] \geq \bigl(1-\eta\bigr)\bigl(1-\frac 1{L+1}\bigr)^{M^{1/2}}.
\]
Since $\eta$ and $L$ are arbitrary, and the equation above holds for all $K_1\geq K^\ast(\eta,L)$, it follows that, for any $M\geq 0$,
\begin{equation}\label{eq:normalized}
\liminf_{K_1\to\infty}  {\mathbb{P}}[\text{ for all } i\leq M^{1/2},\ \pi^{\mathfrak{T}^{(K+K_1)}}(V_{K+K_1+i})\in S^{(K,K_1)}] =1.
\end{equation}

Now we turn our focus to the discrete.
We can notice that the event $\{\text{For all } i\leq M^{1/2},\ \pi^{T^{(n,K+K_1)}}(V^n_{K+K_1+i})\in S^{(n,K,K_1)}\}$ is measurable with respect the shape of $T^{(n,K+K_1+M^{1/2})}$  (recall the definition of shape in Section~\ref{sect_condG}). This leads us to use condition $(G)_{\sigma_d,\sigma_{\phi}}^h$ in order to show that, with probability going to 1 as $n$ goes to infinity, $T^{(n,K+K_1+M^{1/2})}$ and $T^{(K+K_1+M^{1/2})}$ have the same shape, which implies
\begin{align*}
& \lim_{n \to \infty}{\mathbb P}[\text{ for all } i\leq M^{1/2},\ \pi^{T^{(n,K+K_1)}}(V^n_{K+K_1+i})\in S^{(n,K,K_1)}] \\
=& {\mathbb P}[\text{ for all } i\leq M^{1/2},\ \pi^{\mathfrak{T}^{(K+K_1)}}(V_{K+K_1+i})\in S^{(K,K_1)}] \\
\geq & \bigl(1-\frac 1M\bigr)\exp\bigl(-\frac 2{M^{1/2}}\bigr),
\end{align*}
where we used the previous equation.  This proves the lemma for condition $(G)_{\sigma_d,\sigma_\phi}$. 

In order to get the lemma under condition $(G)_{\sigma_d,\sigma_\phi}^h$, it suffices to obtain \eqref{eq:normalized} for $\frak{T}^h$, the CRT conditioned on having height larger than $h$,  instead of $\frak{T}$. The rest of the proof is completely analogous.
It will be helpful to write $\frak{T}^h$ as a mixture of CRTs conditioned on a fixed volume and on having height greater than $h$. More precisely, let $\frak{T}_v$ be the CRT conditioned on having volume exactly equal to $v$. This is the same as conditioning the underlying Brownian excursion to have duration $v$. Moreover, $\frak{T}_v$ can be obtained from the normalized CRT via scaling, i.e., $\frak{T}_v=(\frak{T},\sqrt{v}d_{\frak{T}},v\lambda_{\frak{T}})$.
Now, let $\frak{T}_v^h$ be $\frak{T}_v$ conditioned on having height greater than $h$. Since, for any $v\geq0$, $\mathbb{P}[h(\frak{T}_v)\geq h ]>0$, this conditioning can be defined in the standard way
\[\mathbb{P}[\frak{T}_v^h\in A]:=\frac{\mathbb{P}[\frak{T}_v\in A,h(\frak{T}_v)\geq h]}{\mathbb{P}[h(\frak{T}_v)\geq h]}\]
for any measurable set $A$. The measurability of $A$ means that it is a Borelian subset of the set of ordered trees that can be constructed using a continuous excursion, as in Section \ref{sect:crt}, endowed with the topology inherited from the uniform convergence of the corresponding excursions.
Let $\nu$ be the distribution of the total volume of $\frak{T}^h$. The decomposition of $\frak{T}^h$ we were after is 
\begin{equation}\label{eq:mixture}
\mathbb{P}\left[\frak{T}^h\in A \right]=\int_0^\infty \mathbb{P}[\frak{T}_v^h\in A] \nu(dv),
\end{equation}
for all measurable $A$.
  
It follows from \eqref{eq:normalized} and the scaling relation between $\frak{T}$ and $\frak{T}_v$ that
\begin{equation}
\liminf_{K_1\to\infty}  {\mathbb P}[\text{ for all } i\leq M^{1/2},\ \pi^{\mathfrak{T}_v^{(K+K_1)}}(V_{K+K_1+i})\in S^{(K,K_1)}] =1
\end{equation}
Also, using that $\frak{T}^h_v$ is defined as $\frak{T}_v$ conditioned on a positive probability event, we get from the display above  that
\begin{equation}\label{eq:height}
\liminf_{K_1\to\infty}  {\mathbb P}[\text{ for all } i\leq M^{1/2},\ \pi^{\mathfrak{T}_v^{h,K+K_1}}(V_{K+K_1+i})\in S^{(K,K_1)}] =1.
\end{equation}
Finally, using \eqref{eq:mixture}, we can deduce from the above that
\begin{equation}\label{eq:forheightinthelimit}
\liminf_{K_1\to\infty}  {\mathbb P}[\text{ for all } i\leq M^{1/2},\ \pi^{\mathfrak{T}^{h,K+K_1}}(V_{K+K_1+i})\in S^{(K,K_1)}] =1.
\end{equation}
We finish the proof for condition $(G)_{\sigma_d,\sigma_\phi}^h$ in the same way that we did for condition $(G)_{\sigma_d,\sigma_\phi}$, using the display above instead of \eqref{eq:normalized}.
 \end{proof}

We will now prove that most of the edge-volume of $G_n$ projected on $\T^{(n,K+K_1)}$ is projected on $S^{(n,K,K_1)}$ (i.e.~on the $K_1$ last branches).
\begin{lemma}\label{lem_all_mass}
Assume that either of conditions $(G)_{\sigma_d,\sigma_{\phi}}$ or $(G)_{\sigma_d,\sigma_{\phi}}^h$ are verified with points chosen according to uniform edge-volume. Fix $K<\infty$, then for any $\epsilon,\epsilon'>0$, there exists $C_1(\epsilon,\epsilon',K)<\infty$ such that for any $K_1\geq C_1$ we have
\begin{align*}
& \liminf_{n\to \infty} {\mathbb P}[\left|\{e \in E(G_n), \pi^{(n,K+K_1)}(e)\in S^{(n,K,K_1)}\}\right| \geq (1-\epsilon) |E(G_n)|] \\
\geq & 1-\epsilon'.
\end{align*}
\end{lemma}
\begin{proof}
Fix $\epsilon, \epsilon'>0$. Let $B\subset V(G_n)$ and recall from the beginning of Section \ref{sec:simscalimthm} the definition of the subgraph $B^*$. Let us denote the event that
\begin{equation}\label{eq:defofa}
A_{\epsilon}(n):=\Bigl\{\max_{\substack{B\subset V(G_n) \\ B^* \text{connected}}} \abs{{\mathbb P}[V_1^n \in B\mid G_n] - \frac{\abs{E(B^*)}}{\abs{E(G_n)}}}<\frac{\epsilon}{2}\Bigr\},
\end{equation}
which occurs with high probability for $n$ large enough since the $V_i^n$ are chosen according to edge volume.

Recall that the points $V_i^n$ are chosen in an i.i.d.~fashion. We denote $\tilde{S}^{(n,K,K_1)}:=\{x\in G_n,\ \pi^{(n,K+K_1)}(x)\in S^{(n,K,K_1)}\}$. We know that $\tilde{S}^{(n,K,K_1)}$ has at most $K_1$ connected components, so we can notice that on $A_{ \epsilon/K_1}$ we have that $\abs{{\mathbb P}[V_{K+K_1+i}^n \in \tilde{S}^{(n,K,K_1)}\mid G_n]-\frac{|E[\tilde{S}^{(n,K,K_1)}]|}{|E(G_n)|}} \leq \epsilon/2$ for any $i\geq 0$.

Hence, on the event $A_{\epsilon/K_1}(n)$ (which is measurable with respect to $G_n$), we have for any $i\geq 0$
\[
{\mathbb P}[V^n_{K+K_1+i}\in \tilde{S}^{(n,K,K_1)}\mid G_n,\ (V^n_i)_{i\leq K+K_1}]\leq \frac{|E(\tilde{S}^{(n,K,K_1)})|}{|E(G_n)|}+\frac{\epsilon}2,
\]
and since the points $V^n_{K+K_1+i}$ are chosen independently, we have that, on  $A_{\epsilon/K_1}(n)$,
\begin{align*}
& {\mathbb P}[\text{ for all } i\leq M^{1/2},\ \pi^{T^{(n,K+K_1)}}(V^n_{K+K_1+i})\in S^{(n,K,K_1)} \mid G_n,\ (V^n_i)_{i\leq K+K_1}] \\
=&  \Bigl(\frac{|E(\tilde{S}^{(n,K,K_1)})|}{|E(G_n)|}+\frac{\epsilon}2 \Bigr)^{M^{1/2}},
\end{align*}
by the second point of Remark~\ref{rem_prop_def}.

Let us denote 
\[
B_{\epsilon}(n):=\left\{\left|\{e \in E(G_n), \pi^{(n,K+K_1)}(e)\in S^{(n,K,K_1)}\}\right| \leq (1-\epsilon)\left|E(G_n)\right|\right\},
\]
then, we can see that
\begin{align*}
& {\mathbb P}[\text{ for all } i\leq M^{1/2},\ \pi^{(n,K+K_1)}(V^n_{K+K_1+i})\in S^{(n,K,K_1)} \mid  A_{\epsilon/K_1}(n) \cap B_{\epsilon}(n)] \\ \leq & (1-\epsilon/2)^{M^{1/2}} \leq \exp(-\epsilon M^{1/2}/2).
\end{align*}

Putting together the previous elements, we can see that
\begin{align*}
 {\mathbb P}\left[B_\epsilon(n)\right]\leq& {\mathbb P}[\text{ for some } i\leq M^{1/2},\ \pi^{(n,K+K_1)}(V^n_{K+K_1+i})\notin S^{(n,K,K_1)}]\\
&+ {\mathbb P}[B_\epsilon(n);\text{ for all } i\leq M^{1/2},\ \pi^{(n,K+K_1)}(V^n_{K+K_1+i})\in S^{(n,K,K_1)}]\\
\leq& {\mathbb P}[\text{ for some } i\leq M^{1/2},\ \pi^{(n,K+K_1)}(V^n_{K+K_1+i})\notin S^{(n,K,K_1)}]+{\mathbb P}[ A_{\epsilon/K_1}(n)^c]\\
&+ {\mathbb P}[B_\epsilon(n)\cap A_{\epsilon/K_1}(n);\text{ for all } i\leq M^{1/2},\ \pi^{(n,K+K_1)}(V^n_{K+K_1+i})\in S^{(n,K,K_1)}]\\
\leq& {\mathbb P}[\text{ for some } i\leq M^{1/2},\ \pi^{(n,K+K_1)}(V^n_{K+K_1+i})\notin S^{(n,K,K_1)}]+{\mathbb P}[ A_{\epsilon/K_1}(n)^c]\\
&+ {\mathbb P}[\text{ for all } i\leq M^{1/2},\ \pi^{(n,K+K_1)}(V^n_{K+K_1+i})\in S^{(n,K,K_1)}\vert B_\epsilon(n)\cap A_{\epsilon/K_1}(n)]\\
\leq & {\mathbb P}[\text{ for some } i\leq M^{1/2},\ \pi^{(n,K+K_1)}(V^n_{K+K_1+i})\notin S^{(n,K,K_1)}] \\
   & + {\mathbb P}[ A_{\epsilon/K_1}(n)]+\exp(-\epsilon M^{1/2}/2),
   \end{align*}
where the second term goes to 0 by Definition~\ref{def_edge_unif}, and then by Lemma~\ref{lem_step0} we know that  for any $M$ there exists $K_1$ large enough  such that
\begin{align*}
& \limsup_{n \to \infty} {\mathbb P}[B_\epsilon(n)]  \leq  1- (1-\frac 1M)\exp(-\frac 2{M^{1/2}}) +\exp(-\epsilon M^{1/2}/2),
   \end{align*}
   and we can see that for any $\epsilon'>0$ there exists an $M$ large enough (provided we take $K_1$ large enough) such that right-hand side is lower than $\epsilon'$. This proves the lemma. 
\end{proof}

\subsubsection{Assigning a weight to each atom of $\hat{\mu}^{(n,K,K_1,K_2)}$}Our goal in this section is to estimate the amount of mass that should be put to each atom $y^{(n,K,K_1)}_i$ of $\hat{\mu}^{(n,K)}$ of $\hat{\mu}^{(n,K,K_1,K_2)}$. We want to do this using information depending only on $\mathcal{T}^{(n,K')}$ for some large $K'$.

Set 
\[
\hat{\mu}^{(n,K,K_1,K_2)}(y^{(n,K,K_1)}_i)=\frac{\left|\{j \leq K_2,\  \pi^{(n,K)}(V^n_{K+K_1+j})=y^{(n,K,K_1)}_i\}\right|}{K_2}.
\]
and recall that for any $y\in V^*(\T^{(n,K)})$
\[
\mu^{(n,K)}(y)=\frac{\left|\{e\in E(G_n),\ \pi^{(n,K)}(e)=y\}\right|}{|E(G_n)|}.
\]

In essence we span $K_1$ point to create the locations of the atoms and then an extra $K_2$ points to assign a weight to those atoms.

\begin{lemma}\label{lem_approx_branch}
Assume that either of conditions $(G)_{\sigma_d,\sigma_{\phi}}$ or $(G)_{\sigma_d,\sigma_{\phi}}^h$ are verified with points chosen according to uniform edge-volume.
For any $K,K_1<\infty$ and $\epsilon,\epsilon'>0$, there exists $C_2(K,K_1,\epsilon,\epsilon')$ such that for any $K_2\geq C_2$ we have
\begin{align*}
 &\lim_{n\to \infty} {\mathbb P}\Bigl[\text{ for all $i\leq l^{(K,K_1)}$},\ \abs{ \mu^{(n,K)}(y^{(n,K,K_1)}_i)-\hat{\mu}^{(n,K,K_1,K_2)}(y^{(n,K,K_1)}_i) }<\frac{\epsilon}{K_1}\Bigr]\\
\geq &1-\epsilon'.
\end{align*}
\end{lemma}

\begin{proof}


Recall the definition of $A_{\epsilon}(n)$ from \eqref{eq:defofa}.

Denote $\tilde{s}_i^{(n,K,K_1)}:=\{x\in G_n,\ \pi^{(n,K+K_1)}(x)\in s_i^{(n,K,K_1)}\}$. Since for all  $i\leq l^{(K,K_1)}$ the set $\tilde{s}_i^{(n,K,K_1)}$ is connected we know that  on $A_{\epsilon}(n)$ we have $\abs{{\mathbb P}[V_1^n \in \tilde{s}_i^{(n,K,K_1)}\mid G_n] - \frac{\abs{E(\tilde{s}_i^{(n,K,K_1)})}}{\abs{E(G_n)}}}<\frac{\epsilon}2$.  Now, by the second point of Remark~\ref{rem_prop_def}, $\mu^{(n,K)}(y_i^{(n,K,K_1)})=\frac{\abs{E(\tilde{s}_i^{(n,K,K_1)})}}{\abs{E(G_n)}}$, this means that, on $A_{\epsilon}(n)$,
\[
\abs{{\mathbb P}[V_{K+K_1+j}^n \in \tilde{s}_i^{(n,K,K_1)}\mid G_n,(V^n_j)_{j\leq K+K_1}] - \mu^{(n,K)}(y_i^{(n,K,K_1)})}<\frac{\epsilon}2 \]
 for any $j\geq 1$.

For any $1\leq i\leq l^{(K,K_1)}$, let us notice that, conditionally on $G_n$ and  $(V^n_j)_{j\leq K+K_1}$, we have that $\hat{\mu}^{(n,K,K_1,K_2)}(y_i^{(n,K,K_1)})=\frac 1 {K_2} \sum_{i=1}^{K_2} B_i$, where $B_i$ are i.i.d.~random variables such that $P[B_i=1]=1-P[B_i=0]= {\mathbb P}[V_{K+K_1+i}\in \tilde{s}_i^{(n,K,K_1)}\mid G_n, (V^n_j)_{j\leq K+K_1}]$ (here we used the first point of Remark~\ref{rem_prop_def}). This means, by Markov's inequality, for all $i\leq l^{(K,K_1)}$
\begin{align*}
& {\mathbb P}\Bigl[\abs{\hat{\mu}^{(n,K+K_1+K_2)}(y_i^{(n,K,K_1)})-\mu^{(n,K)}(y_i^{(n,K,K_1)})}>\epsilon\Big| A_{\epsilon}(n)\Bigr]\\
\leq & {\mathbb P}\Bigl[\abs{\hat{\mu}^{(n,K+K_1+K_2)}(y_i^{(n,K,K_1)})-{\mathbb P}[V_{K+K_1+i}^n\in \tilde{s}_i^{(n,K,K_1)}\mid G_n, (V^n_j)_{j\leq K+K_1}]}>\epsilon/2\Big| A_{\epsilon}(n)\Bigr]\\
\leq & \frac 4{\epsilon^2 K_2}.
\end{align*}

We can use the previous computation to see that
\begin{align*}
&  {\mathbb P}\Bigl[\exists i\leq l^{(K,K_1)},\ \abs{ \mu^{(n,K)}(y_i^{(n,K,K_1)})-\hat{\mu}^{(n,K,K_1,K_2)}(y_i^{(n,K,K_1)}) }>\frac{\epsilon}{K_1}\Bigr] \\
\leq & {\mathbb E}\Bigl[\sum_{i=1}^{l^{(K,K_1)}} {\mathbb P}\Bigl[\abs{ \mu^{(n,K)}(y_i^{(n,K,K_1)})-\hat{\mu}^{(n,K,K_1,K_2)}(y_i^{(n,K,K_1)}) }>\frac{\epsilon}{K_1}\mid A_{\epsilon}(n)\Bigr]\Bigr] \\
& \qquad \qquad +{\mathbb P}[A_{\epsilon}(n)^c]\\
\leq &  {\mathbb E}\Bigl[\sum_{i=1}^{l^{(K,K_1)}} \frac {4K_1^2}{\epsilon^2 K_2}\Bigr] +{\mathbb P}[A_{\epsilon}(n)^c] \\
\leq &\frac {4K_1^3}{\epsilon^2 K_2} +{\mathbb P}[A_{\epsilon}(n)^c],
\end{align*}
since $l^{(K,K_1)}\leq K_1$. By taking $n$ to infinity the second probability goes to 0 by Definition~\ref{def_edge_unif} and then we choose $K_2$ large enough to make this quantity smaller than $\epsilon'$, which proves the lemma.
\end{proof}

\subsection{Conclusion}

By construction the measure $\hat{\mu}^{(n,K,K_1,K_2)}$ is measurable with respect to $\T^{(n,K+K_1+K_2)}$ and has support on $\T^{(n,K)}$. 
Its other key property can be obtained by combining Lemma~\ref{lem_all_mass} and Lemma~\ref{lem_approx_branch} to prove the following. Let $d_{TV}$ denote the total variation distance between measures.
\begin{lemma}\label{lem_mu_approx}
Assume that either of conditions $(G)_{\sigma_d,\sigma_{\phi}}$ or $(G)_{\sigma_d,\sigma_{\phi}}^h$ are verified with points chosen according to uniform edge-volume.
For any $\epsilon,\epsilon'>0$ and $K<\infty$ there exists $C_1(\epsilon,\epsilon',K)<\infty$, such that, for any $K_1\geq C_1$ there exists $C_2(\epsilon,\epsilon',K,K_1)<\infty$, such that, for any $K_2\geq C_2 $ we have that
\[
\limsup_{n\to \infty} {\mathbb P}[d_{TV}(\hat{\mu}^{(n,K,K_1,K_2)},\mu^{(n,K)})\geq \epsilon]\leq \epsilon'.
\]
\end{lemma}
\begin{proof}
Fix $\epsilon, \epsilon'>0$. Let $A$ be a subset of $\T^{(n,K)}$, then
\begin{align*}
\hat{\mu}^{(n,K,K_1,K_2)}(A)=& \hat{\mu}^{(n,K,K_1,K_2)}(A\setminus (\cup_{i=1}^{l^{(n,K,K_1)}} y_i^{(n,K,K_1)})) \\& \qquad +\sum_{i=1}^{l^{(n,K,K_1)}} {\bf 1}_{\{y_i^{(n,K,K_1)}\in A\}} \hat{\mu}^{(n,K,K_1,K_2)}(y_i^{(n,K,K_1)}),
\end{align*}
with a similar formula holding for $\mu^{(n,K)}$. This means that the total variation
\begin{align*}
&\max_{A\subset \T^{(n,K)}} \abs{\hat{\mu}^{(n,K,K_1,K_2)}(A)-\mu^{(n,K)}(A)} \\
\leq &\max_{A\subset \T^{(n,K)}}  \abs{\hat{\mu}^{(n,K,K_1,K_2)}(A\setminus (\cup_{i=1}^{l^{(n,K,K_1)}} y_i^{(n,K,K_1)}))-\mu^{(n,K)}(A\setminus (\cup_{i=1}^{l^{(n,K,K_1)}} y_i^{(n,K,K_1)}))} \\
& + \sum_{i=1}^{l^{(n,K,K_1)}} \abs{\hat{\mu}^{(n,K,K_1,K_2)}(y_i^{(n,K,K_1)})-\mu^{(n,K)}(y_i^{(n,K,K_1)})}.
\end{align*}

Firstly,  by Lemma~\ref{lem_approx_branch} and the fact that $l^{(n,K,K_1)}\leq K_1$ we know that
\[
\limsup_{n\to \infty} {\mathbb P}\Bigl[\sum_{i=1}^{l^{(n,K,K_1)}} \abs{\hat{\mu}^{(n,K,K_1,K_2)}(y_i^{(n,K,K_1)})-\mu^{(n,K)}(y_i^{(n,K,K_1)})} \geq \epsilon\Bigr]\leq \epsilon',
\]

Secondly, we have $\hat{\mu}^{(n,K,K_1,K_2)}(A\setminus (\cup_{i=1}^{l^{(n,K,K_1)}} y_i^{(n,K,K_1)}))=0$, and by Lemma~\ref{lem_all_mass} we know that
\[
\limsup_{n\to \infty} {\mathbb P}\Bigl[\max_{A\subset \T^{(n,K)}}\mu^{(n,K)}(A\setminus (\cup_{i=1}^{l^{(n,K,K_1)}} y_i^{(n,K,K_1)}))\geq \epsilon \Bigr]\leq \epsilon'.
\]

Combining the last three equations we can see that
\[
\limsup_{n\to \infty} {\mathbb P}\Bigl[d_{TV}(\hat{\mu}^{(n,K,K_1,K_2)},\mu^{(n,K)})\geq 2\epsilon\Bigr]\leq 2\epsilon',
\]
which implies the result.
\end{proof}

\section{Condition $(G)$ implies condition $(V)$.}\label{sect_gv2}

We are now going to use the empirical measure of the previous section to deduce condition $(V)$ from condition $(G)_{\sigma_d,\sigma_{\phi}}$ or from condition $(G)_{\sigma_d,\sigma_{\phi}}^h$. 

\subsection{Extending the previous proofs to the continuous case}

The proof of Lemma~\ref{lem_all_mass} (which is done in the discrete setting) can be adapted to the continuous case and, recalling the definitions from Remark \ref{rmk:defofS}, we obtain 
\begin{lemma}\label{lem_all_mass_cont}
For any $\epsilon,\epsilon'>0$ and $K<\infty$ there exists $C_1(\epsilon, \epsilon',K)<\infty$ such that for any $K_1\geq C_1$ we have
\[
{\mathbb P}[\lambda_\frak T(x\in \mathfrak{T},\ \pi^{\mathfrak{T}^{(K+K_1)}}(x)\in S^{(\mathfrak{T},K,K_1)}) \geq 1-\epsilon ] 
\geq  1-\epsilon'
\]
and
\[
{\mathbb P}\left[\lambda_{\frak T^h}^1(x\in \mathfrak{T}^h,\ \pi^{\mathfrak{T}^{h,K+K_1}}(x)\in S^{(\mathfrak{T},K,K_1)}) \geq 1-\epsilon \right] 
\geq  1-\epsilon',
\]
where $\lambda_{\frak T^h}^1$ denotes the normalization of $\lambda_{\frak T^h}$. That is $\lambda_{\frak T^h}^1(\cdot)=\lambda_{\frak T^h}(\frak T^h)^{-1}\lambda_{\frak T^h}(\cdot)$. 
\end{lemma}

Defining 
\[
\hat{\mu}^{(\mathfrak{T},K,K_1,K_2)}(y^{(\mathfrak{T},K,K_1)}_i)=\frac{\left|\{j \leq K_2,\  \pi^{\mathfrak{T}^{(K)}}(V_{K+K_1+j})=y^{(\mathfrak{T},K,K_1)}_i\}\right|}{K_2}.
\]
and $\hat{\mu}^{(\mathfrak{T}^h,K,K_1,K_2)}$ in an analog fashion, with $\frak T^h$ in place of $\frak T$.  
we can also obtain adapt the proof of Lemma~\ref{lem_mu_approx} to the continuous setting to get
\begin{lemma}\label{lem_mu_approx_cont}
For any $\epsilon,\epsilon'>0$ and $K<\infty$ there exists $C_1(\epsilon,\epsilon',K)<\infty$ such that for any $K_1\geq C_1$ there exists $C_2(\epsilon,\epsilon',K,K_1)<\infty$ such that for any $K_2\geq C_2 $ we have that 
\[
 {\mathbb{P}}[d_{TV}(\hat{\mu}^{(\mathfrak{T},K,K_1,K_2)},\mu^{(\mathfrak{T},K)})\geq \epsilon]\leq \epsilon'
\]
and
\[
 {\mathbb{P}}[d_{TV}(\hat{\mu}^{(\mathfrak{T}^h,K,K_1,K_2)},\mu^{(\mathfrak{T}^h,K)})\geq \epsilon]\leq \epsilon'
.\]
\end{lemma}

The only significant difference is that the proofs are simpler in this case because we work with uniform points and not points chosen according to uniform edge-volume (which means we can forgo  the events $A_{\epsilon}(n)$).

\subsection{A preliminary result}

By condition $(G)_{\sigma_d,\sigma_{\phi}}$ and the Skorohod representation Theorem, we can assume that, for all $K\in\N$, the augmented random graphs $(G_n, (V^n_i)_{i = 1,\dots,K})_{n\in\N}$ and $\mathfrak{T}^{(K)}$ are defined in a common probability space $(\Omega^{(K)},\mathcal{F}^{(K)},{\mathbf{P}}^{(K)})$ and that  
\begin{equation}\label{eq:couplinggeometry}
(T^{(n,K)},d^{(n,K)},n^{-1/2}\phi_{\T_n})\stackrel{n\to\infty}{\to}(\mathfrak{T}^{(K)}, \sigma_d d_{\mathfrak{T}^{(K)}},\sigma_{\phi}\sqrt{\sigma_d}\phi_{\mathfrak{T}^{(K)}}) \qquad {\mathbf{P}}^{(K)}\text{-a.s.},
\end{equation}
in the topology induced by $D$, where $D$ is the natural topology on graph spatial trees (see display \eqref{eq:defofD}).

Similarly, under condition $(G)_{\sigma_d,\sigma_{\phi}}^h$ we can assume that
\begin{equation}\label{eq:couplinggeometryh}
(T^{(n,K)},d^{(n,K)},n^{-1/2}\phi_{\T_n})\stackrel{n\to\infty}{\to}(\mathfrak{T}^{h,K}, \sigma_d d_{\mathfrak{T}^{h,K}},\sigma_{\phi}\sqrt{\sigma_d}\phi_{\mathfrak{T}^{h,K}}) \qquad {\mathbf{P}}^{(K)}\text{-a.s.}.
\end{equation}

By condition $(G)_{\sigma_d,\sigma_\phi}$ (resp, $(G)_{\sigma_d,\sigma_\phi}^h$), we know that $\frak{T}^{(n,K)}$ and $\frak{T}^{(K)}$ (resp. $\frak{T}^{h,K}$) are homeomorphic for $n$ large enough, and moreover that the homeomorphism can be chosen to preserve lexicographical order.
  Let, $\Upsilon_{n,K}:\frak{T}^{(n,K)}\to\frak{T}^{(K)}$ be defined as
  \begin{equation}\label{eq:defofupsilon}
  \Upsilon_{n,K}=\Upsilon_{\frak{T}^{(n,K)},\frak{T}^{(K)}},
  \end{equation} where, $\Upsilon_{\frak{T}^{(n,K)},\frak{T}^{(K)}}$ is as in Section \ref{sect_condG}.

Let us prove the following
\begin{lemma}\label{a_lem_tech}
If a sequence of random augmented graph verifiy condition $(G)_{\sigma_d,\sigma_{\phi}}$, then we can see that for any $\epsilon>0$ and $K>0$, 
\[
\limsup_{n\to \infty} {\bf P}^{(K)}\Bigl[ \sup_{x\in \T^{(n,K)}} \abs{\lambda^{(n,K)}(\overrightarrow{\T_x^{(n,K)}}) -\lambda^{(\mathfrak{T},K)}(\overrightarrow{\mathfrak T_{\Upsilon_{n,K}(x)}^{(K)}})} >\epsilon\Bigr]=0.
\]
Under condition $(G)_{\sigma_d,\sigma_{\phi}}^h$, the same holds true with $\frak{T}^h$ in place of $\frak T$.
\end{lemma}

\begin{proof}
We will direct the discussion with condition $(G)_{\sigma_d,\sigma_\phi}$ in mind, making the necessary remarks to deal with $(G)^h_{\sigma_d,\sigma_\phi}$ when necessary.
The strategy of the proof is to restrict the supremum appearing in the lemma to a supremum over vertices that are branching points or leaves of some $\T^{(n,K')}$ which will then allow us to use condition $(G)_{\sigma_d,\sigma_{\phi}}$ to show that for large $n$ the distances on $\T^{(n,K)}$ are close to those on $\mathfrak{T}^{(K)}$.
\vspace{0.2cm}

{\it Step 1: Constructing a dense set}

\vspace{0.2cm}

Fix $\delta>0$ and $K<K'\in \N$ and an augmented graph $(G,(V_i)_{i\in \N})$ such that $G(K')$ is tree like (recall definition \ref{def_thin}). We say that $V_0,\ldots V_{K'}$ is $\delta$-dense in $\T^{(G,K)}$ if 
\begin{enumerate}
\item the set $\{\pi^{(G,K)}(V_{l})\text{  with }l \leq K'\}$ has at least one point on every edge of $\mathfrak{T}^{(G,K)}$.
\item If $x$, $y \in \{\pi^{(G,K)}(V_{l})\text{  with } l \leq K'\}$  are neighbours (in  the sense that on the unique  path between them there is no other point in $\{\pi^{(G,K)}(V_{l}) \text{ with }l \leq K'\}$) then there exists $i_x,i_y \leq K'$ such that $\pi^{(G,K)}(V_{i_y})=x$, $\pi^{(G,K)}(V_{i_y})=y$ and $d_{\T^{(G,K')}}(V_{i_x},V_{i_y})\leq \delta$.
\end{enumerate}

This is illustrated in Figure \ref{fig:2}.

Let us recall from Section \ref{sect:crt} that $\Xi$ denotes the law of the CRT and $\lambda_{\mathfrak{T}}$ is the uniform measure in the CRT.
By Lemmas 5.1 and 5.2 in~\cite{BCFa}, we know that for any $\epsilon'>0$ and $\delta>0$, there exists $K'$ such that we have 
\begin{equation}\label{a_delta_blabla1}
\Xi \otimes (\lambda_{\mathfrak{T}})^{\otimes \N}[V_0\ldots V_{K'} \text{ is not $\delta$-dense in }\mathfrak{T}^{(K)}] \leq \epsilon'
\end{equation}
and
\begin{equation}\label{a_delta_blabla2}
  \limsup_{n\to \infty} {\bf P}^{(K)}[V_0^n\ldots V_{K'}^n \text{ is not $\delta$-dense in  }\T^{(n,K)}] \leq \epsilon'.
\end{equation}
Recall that $\lambda^1_{\mathfrak{T}^h}$ is the normalization of $\lambda_{\mathfrak{T}^h}$. The same proof applies for $\frak T^h$ to get that for any $\epsilon'>0$ and $\delta>0$, there exists $K'$ such that we have 
\begin{equation}\label{a_delta_blabla12}
\Xi^h \otimes (\lambda^1_{\mathfrak{T}^h})^{\otimes \N}[V_0\ldots V_{K'} \text{ is not $\delta$-dense in }\mathfrak{T}^{h,K}] \leq \epsilon'
\end{equation}
and equation~\eqref{a_delta_blabla2} also holds under condition $(G)^h_{\sigma_d,\sigma_\phi}$.

\vspace{0.2cm}

{\it Step 2: }

\vspace{0.2cm}

Assume that $V_0^n\ldots V_{K'}^n$ is $\delta$-dense in $\T^{(n,K)}$. Let $x\in \T^{(n,K)}$, and define $\tilde{x}_{n,K,K'}$ in the following manner
\begin{itemize}
\item if $x\in V(T^{(n,K)})$ (i.e.~a leaf or a branch point of $T^{(n,K)}$) then $\tilde{x}_{n,K,K'}=x$,
\item if $x\notin V(T^{(n,K)})$, then $x$ is in exactly one of the edges of $T^{(n,K)}$ (of which there are at most $2K+1$) and we set $\tilde{x}_{n,K,K'}=y$ where $y$ is the first descendant of $x$ in $\T^{(n,K)}$ of the form $\pi^{(n,K)}(V^n_l)$ with $l\leq K'$. 
\end{itemize}

See Figure 3 for an illustration.

  \begin{figure}\label{fig:2}
  \includegraphics[width=0.8\linewidth]{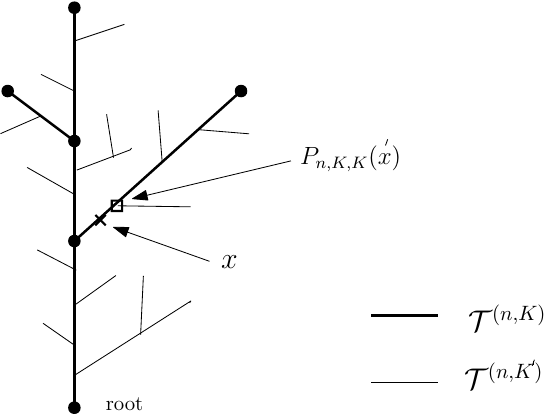}
  \caption{A dense set for $K'$ large enough and an example of $\vec{x}_{n,K,K'}$}
\end{figure}

We emphasize that $\tilde{x}_{n,K,K'}\in V(T^{(n,K')})$ and as such there is only a bounded (at most  $2K'+1$) possibilities for the values of  $\tilde{x}_{n,K,K'}$. Thus, this quantity is suitable for asymptotic analysis using condition $(G)_{\sigma_d,\sigma_{\phi}}$.

We may notice that the definition $\tilde{x}_{n,K,K'}$ implies that for any $x\in \T^{(n,K)}$
\[
\lambda^{(n,K)}(\overrightarrow{\T_x^{(n,K)}}) =\lambda^{(n,K)}(\overrightarrow{\T_{\tilde{x}_{n,K,K'}}^{(n,K)}}) +d^{(n,K)}(x,\tilde{x}_{n,K,K'}),
\]
and by the second point of the definition of a $\delta$-dense set, we see that for any $x\in \T^{(n,K)}$
\[
\abs{\lambda^{(n,K)}(\overrightarrow{\T_x^{(n,K)}}) -\lambda^{(n,K)}(\overrightarrow{\T_{\tilde{x}_{n,K,K'}}^{(n,K)}})}\leq \delta.
\]

This means, using~\eqref{a_delta_blabla2}, that for any $\epsilon',\delta>0$ and for any $K<\infty $, there exists $K'<\infty$ such that
\begin{equation}\label{a_approx_1}
  \limsup_{n\to \infty} {\bf P}^{(K)}\left[\sup_{x\in \T^{(n,K)}} \abs{\lambda^{(n,K)}(\overrightarrow{\T_x^{(n,K)}}) -\lambda^{(n,K)}(\overrightarrow{\T_{\tilde{x}_{n,K,K'}}^{(n,K)}})}> \delta\right] \leq \epsilon'.
  \end{equation}
  
This proof can be extended to the continuous setting (using~\eqref{a_delta_blabla1}) and we find that for any $\epsilon',\delta>0$ and for any $K<\infty $, there exists $K'<\infty$ such that
\begin{equation}\label{a_approx_2}
M \otimes (\lambda^{\mathfrak{T}})^{\otimes \N}\left[\sup_{x\in \mathfrak{T}^{(K)}} \abs{\lambda^{(K)}(\overrightarrow{\mathfrak T_{\Upsilon_{n,K}(x)}^{(K)}}) -\lambda^{(K)}(\overrightarrow{\mathfrak T_{\tilde{x}_{K,K'}}^{(K)}})}> \delta\right] \leq \epsilon',
\end{equation}
  where $\tilde{x}_{K,K'}$ is defined in the same way as $\tilde{x}_{n,K,K'}$ but where all quantities are without script $n$. 
\vspace{0.2cm}

{\it Step 3: Conclusion}

\vspace{0.2cm}
Recall the definition of $\Upsilon_{n,K}$ from the paragraph below display \eqref{eq:couplinggeometry}.
We can see that 
\begin{align*}
&  \sup_{x\in \T^{(n,K)}} \abs{\lambda^{(n,K)}(\overrightarrow{\T_x^{(n,K)}}) -\lambda^{(\mathfrak{T},K)}(\overrightarrow{\mathfrak T_{\Upsilon_{n,K}(x)}^{(K)}})}  \\
\leq & \sup_{x\in \T^{(n,K)}} \abs{\lambda^{(n,K)}(\overrightarrow{\T_x^{(n,K)}}) -\lambda^{(n,K)}(\overrightarrow{\T_{\tilde{x}_{n,K,K'}}^{(n,K)}})}\\ & + \sup_{x\in \T^{(n,K)}} \abs{\lambda^{(n,K)}(\overrightarrow{\T_{\tilde{x}_{n,K,K'}}^{(n,K)}})-\lambda^{(\mathfrak{T},K)}(\overrightarrow{\mathfrak T_{\widetilde{\Upsilon_{n,K}(x)}_{K,K'}}^{(K)}})}
\\& +  \sup_{x\in \T^{(n,K)}} \abs{\lambda^{(\mathfrak{T},K)}(\overrightarrow{\mathfrak T_{\widetilde{\Upsilon_{n,K}(x)}_{K,K'}}^{(K)}})-\lambda^{(\mathfrak{T},K)}(\overrightarrow{\mathfrak{T}_{\Upsilon_{n,K}(x)}^{(K)}})}.
\end{align*}

Therefore, using ~\eqref{a_approx_1} and~\eqref{a_approx_2} with $\delta=\epsilon/4$, we can see that  for some $K'$ large enough
\begin{align} \label{a_fin1}
& \limsup_{n\to \infty} {\bf P}^{(K)}\Bigl[ \sup_{x\in \T^{(n,K)}} \abs{\lambda^{(n,K)}(\overrightarrow{\T_x^{(n,K)}}) -\lambda^{(\mathfrak{T},K)}(\overrightarrow{\mathfrak T_x^{(K)}})} >\epsilon] \\ \nonumber
\leq & 2\epsilon' +   \limsup_{n\to \infty} {\bf P}^{(K)}\left[ \sup_{x\in \T^{(n,K)}} \abs{\lambda^{(n,K)}(\overrightarrow{\T_{\tilde{x}_{n,K,K'}}^{(n,K)}})-\lambda^{(\mathfrak{T},K)}(\overrightarrow{\mathfrak T_{\widetilde{\Upsilon_{n,K}(x)}_{K,K'}}^{(K)}})}>\epsilon/2\right].
\end{align}

We are going to show that $\widetilde{\Upsilon_{n,K}(x)}_{K,K'}=\Upsilon_{n,K}(\tilde{x}_{n,K,K'})$ for any $x\in  \T^{(n,K)}$.  There are two cases
\begin{itemize}
\item $x=V_i^n$ for some $i\leq K$: In this case, $\widetilde{V_i^n}_{n,K,K'}=V_i^n$, $\Upsilon_{n,K}(V_i^n)=V_i$ and $\widetilde{V_i}=V_i$ by the definitions of $\tilde{x}$ and $\Upsilon_{n,K}$. In particular this yields $\widetilde{\Upsilon_{n,K}(V_i^n)}_{K,K'}=\Upsilon_{n,K}(\widetilde{V_i^n}_{n,K,K'})$.
\item $x \in \T^{(n,K)} \setminus \{V_i^n,\ i\leq K\}$:  We denote $e_x^{K'}$ the unique edge $[e_-,e^+]$  in $E(T^{(n,K')})$ (where $e_-\prec e_+$) such that $x$ lies on the path between $e_-$ and $e_+$. We know that $\Upsilon_{n,K}(e_x^{K'})$ is an edge of $E(\mathfrak{T}^{(K')})$ which we write $[f_-,f_+]$ (where $f_-\prec f_+$). Using those notations, we can see that $\tilde{x}_{n,K,K'}=e^+$, so that $\Upsilon_{n,K}(\tilde{x}_{n,K,K'})=\Upsilon_{n,K}(e^+)=f^+$, and on the other hand, since $\Upsilon_{n,K}(x) \in \Upsilon_{n,K}(e_x^{K'})$ (and $\Upsilon_{n,K}(x)\neq f_-$ since $x \neq e_-$) we know that $\widetilde{\Upsilon_{n,K}(x)}_{K,K'}=f^+$. Hence we also have in this case that $\widetilde{\Upsilon_{n,K}(x)}_{K,K'}=\Upsilon_{n,K}(\tilde{x}_{n,K,K'})$.
\end{itemize}

This means that the remaining unknown term of~\eqref{a_fin1} can be upper-bounded in the following manner
\begin{align} \label{a_fin2}
&  \limsup_{n\to \infty} {\bf P}^{(K)}\Bigl[ \sup_{x\in \T^{(n,K)}} \abs{\lambda^{(n,K)}(\overrightarrow{\T_{\tilde{x}_{n,K,K'}}^{(n,K)}})-\lambda^{(\mathfrak{T},K)}(\overrightarrow{\mathfrak T_{\widetilde{\Upsilon_{n,K}(x)}_{K,K'}}^{(K)}})}>\epsilon/2\Bigr] \\ \nonumber
= &  \limsup_{n\to \infty} {\bf P}^{(K)}\Bigl[ \sup_{x\in \T^{(n,K)}} \abs{\lambda^{(n,K)}(\overrightarrow{\T_{\tilde{x}_{n,K,K'}}^{(n,K)}})-\lambda^{(\mathfrak{T},K)}(\overrightarrow{\mathfrak T_{\Upsilon_{n,K}(\tilde{x}_{n,K,K'})}^{(K)}})}>\epsilon/2\Bigr] \\ \nonumber
\leq  & \limsup_{n\to \infty} {\bf P}^{(K)}\Bigl[ \sup_{x\in V(T^{(n,K')})} \abs{\lambda^{(n,K)}(\overrightarrow{\T_{x}^{(n,K)}})-\lambda^{(\mathfrak T,K)}(\overrightarrow{\mathfrak T_{\Upsilon_{n,K}(x)}^{(K)}})}>\epsilon/2\Bigr],
\end{align}
where we used that $\tilde{x}_{n,K,K'} \in V( T^{(n,K')})$.
Now, we can notice that the first event on the right hand side of the equation above is measurable with respect to the shape of $T^{(n,K')}$ and the intrinsic distances between $V(T^{(n,K')})$. Thus by condition $(G)_{\sigma_d,\sigma_{\phi}}$ we can see that 

\begin{equation}  \label{a_fin3}
\limsup_{n\to \infty} {\bf P}^{(K)}\Bigl[ \sup_{x\in V(T^{(n,K')})} \abs{\lambda^{(n,K)}(\overrightarrow{\T_{x}^{(n,K)}})-\lambda^{(\mathfrak T,K)}(\overrightarrow{\T_{\Upsilon_{n,K}(x)}^{(n,K)}})}>\epsilon/2\Bigr]=0,
 \end{equation}
 where we emphasize that the constants $\sigma_d$ and $\sigma_{\phi}$ of condition $(G)_{\sigma_d,\sigma_{\phi}}$ do not appear because the previous equation contains no information about the embedding (hence no constant $\sigma_{\phi}$) and $\lambda^{(n,K)}$ and $\lambda^{(K)}$ are normalized to have total mass 1 (hence the multiplicative constant $\sigma_d\neq 0$ appearing in the distance scaling is irrelevant).
 
 Using~\eqref{a_fin1}, \eqref{a_fin2} and~\eqref{a_fin3}, we can see that 
 \[
  \limsup_{n\to \infty} {\bf P}^{(K)}\Bigl[ \sup_{x\in \T^{(n,K)}} \abs{\lambda^{(n,K)}(\overrightarrow{\T_x^{(n,K)}}) -\lambda^{(\mathfrak{T},K)}(\overrightarrow{\T_x^{(K)}})} >\epsilon]  \leq 2\epsilon',
  \]
 and since the left-hand probability does not depend on $K'$, this equation is valid for any $\epsilon'>0$ and in particular by letting $\epsilon'$ go to $0$
we get the lemma.
\end{proof}

\subsection{Proof of condition (V)}

Our goal is to prove that
\begin{proposition}\label{prop_cond_V}
A sequence of random augmented graphs $(G_n,(V^n_i)_{i\in\N})_{n\in\N}$ verifying either conditions $(G)_{\sigma_d,\sigma_{\phi}}$ or $(G)_{\sigma_d,\sigma_{\phi}}^h$ with points $V_i^n$ chosen according to uniform edge-volume also verifies condition $(V)$. More precisely, we have
\[
 \limsup_{K\to \infty}\limsup_{n\to \infty} {\mathbb P}\Bigl[ \sup_{x\in \T^{(n,K)}} \abs{\lambda^{(n,K)}(\overrightarrow{\T_x^{(n,K)}}) -\mu^{(n,K)}(\overrightarrow{\T_x^{(n,K)}})} >\epsilon\Bigr] =0
 \]
\end{proposition}

\begin{proof}
We will discuss the proof having the case $(G)_{\sigma_d,\sigma_\phi}$ in mind, making the necessary remarks for $(G)_{\sigma_d,\sigma_\phi}^h$ when necessary. Recall the definition of ${\bf P}^{(K)}$ from displays \eqref{eq:couplinggeometry} and \eqref{eq:couplinggeometryh}.
For any $\epsilon>0$, we can use Lemma~\ref{a_lem_tech} to see that
\begin{align*}
& \limsup_{n\to \infty} {\bf P}^{(K)}\Bigl[ \sup_{x\in \T^{(n,K)}} \abs{\lambda^{(n,K)}(\overrightarrow{\T_x^{(n,K)}}) -\mu^{(n,K)}(\overrightarrow{\T_x^{(n,K)}})} >4\epsilon\Bigr] \\
\leq & \limsup_{n\to \infty} {\bf P}^{(K)}\Bigl[ \sup_{x\in \T^{(n,K)}} \abs{\lambda^{(\mathfrak{T},K)}(\overrightarrow{\mathfrak T_{\Upsilon_{n,K}(x)}^{(K)}}) -\mu^{(n,K)}(\overrightarrow{\T_x^{(n,K)}})} >3\epsilon\Bigr]. 
\end{align*}
Under condition $(G)_{\sigma_d,\sigma_\phi}^h$ the same holds true with $\frak T^h$ in place of $\frak T$.

Furthermore, using Lemma~\ref{lem_mu_approx}, for any $\epsilon'>0$ there exists $K_1,K_2<\infty$ such that
\begin{align*}
& \limsup_{n\to \infty} {\bf P}^{(K)}\Bigl[ \sup_{x\in \T^{(n,K)}} \abs{\lambda^{(\mathfrak{T},K)}(\overrightarrow{\mathfrak T_{\Upsilon_{n,K}(x)}^{(K)}}) -\mu^{(n,K)}(\overrightarrow{\T_x^{(n,K)}})} >4\epsilon\Bigr] \\
\leq & \limsup_{n\to \infty} {\bf P}^{(K)}\Bigl[ \sup_{x\in \T^{(n,K)}} \abs{\lambda^{(\mathfrak{T},K)}(\overrightarrow{\mathfrak T_{\Upsilon_{n,K}(x)}^{(K)}}) -\hat{\mu}^{(n,K,K_1,K_2)}(\overrightarrow{\T_x^{(n,K)}})} >2\epsilon\Bigr] +\epsilon'.
\end{align*}
Again, the same holds true with $\frak T^h$ in place of $\frak T$ under condition $(G)_{\sigma_d,\sigma_\phi}^h$.

We can notice that for all $x\in \T^{(n,K)}$ the value of $ \hat{\mu}^{(n,K,K_1,K_2)}(\overrightarrow{\T_x^{(n,K)}})$ only depends on the shape $T^{(n,K+K_1+K_2)}$ (where the definition of shape can be found in Definition~\ref{def_gst}) . This means, that we can use condition $(G)_{\sigma_d,\sigma_{\phi}}$ to see that for any $K,K_1,K_2$ and for any $\epsilon>0$ 

\[
\limsup_{n\to \infty} {\bf P}^{(K)}\Bigl[\text{ for all $x\in \T^{(n,K)}$},\   \hat{\mu}^{(n,K,K_1,K_2)}(\overrightarrow{\T_x^{(n,K)}}) = \hat{\mu}^{(\mathfrak{T},K,K_1,K_2)}(\overrightarrow{\mathfrak T_{\Upsilon_{n,K}(x)}^{(K)}}) \Bigr]=1.
\]

Hence, the three previous equations become
\begin{align*}
&  \limsup_{n\to \infty} {\bf P}^{(K)}\Bigl[ \sup_{x\in \T^{(n,K)}} \abs{\lambda^{(n,K)}(\overrightarrow{\T_x^{(n,K)}}) -\mu^{(n,K)}(\overrightarrow{\T_x^{(n,K)}})} >4\epsilon\Bigr] \\
\leq & \limsup_{n\to \infty} {\bf P}^{(K)}\Bigl[ \sup_{x\in \T^{(n,K)}} \abs{\lambda^{(\mathfrak{T},K)}(\overrightarrow{\mathfrak T_{\Upsilon_{n,K}(x)}^{(K)}}) -\hat{\mu}^{(\mathfrak{T},K,K_1,K_2)}(\overrightarrow{\mathfrak T_{\Upsilon_{n,K}(x)}^{(K)}})} >2\epsilon\Bigr] +\epsilon'.
\end{align*}

However, recalling Lemma~\ref{lem_mu_approx_cont}, we can transform the previous equation into (by possibly increasing $K_1$ and $K_2$)
\begin{align*}
&  \limsup_{n\to \infty} {\bf P}^{(K)}\Bigl[ \sup_{x\in \T^{(n,K)}} \abs{\lambda^{(n,K)}(\overrightarrow{\T_x^{(n,K)}}) -\mu^{(n,K)}(\overrightarrow{\T_x^{(n,K)}})} >3\epsilon\Bigr] \\
\leq & \limsup_{n\to \infty} {\bf P}^{(K)}\Bigl[ \sup_{x\in \T^{(n,K)}} \abs{\lambda^{(\mathfrak{T},K)}(\overrightarrow{\mathfrak T_{\Upsilon_{n,K}(x)}^{(K)}}) -\mu^{(\mathfrak{T},K)}(\overrightarrow{\mathfrak T_{\Upsilon_{n,K}(x)}^{(K)}})  }>\epsilon\Bigr] + 2\epsilon',
\end{align*}
and since the $K_1$, $K_2$ dependence has disappeared, we  can let $\epsilon'$ go to $0$ and see that 
\begin{align*}
&  \limsup_{n\to \infty} {\bf P}^{(K)}\Bigl[ \sup_{x\in \T^{(n,K)}} \abs{\lambda^{(n,K)}(\overrightarrow{\T_x^{(n,K)}}) -\mu^{(n,K)}(\overrightarrow{\T_x^{(n,K)}})} >3\epsilon\Bigr] \\
\leq & {\bf P}^{(K)}\Bigl[ \sup_{x\in \T^{(n,K)}} \abs{\lambda^{(\mathfrak{T},K)}(\overrightarrow{\mathfrak T_{\Upsilon_{n,K}(x)}^{(K)}}) -\mu^{(\mathfrak{T},K)}(\overrightarrow{\mathfrak T_{\Upsilon_{n,K}(x)}^{(K)}})  }>\epsilon\Bigr] .
\end{align*}
The same holds true with $\frak T^h$ in place of $\frak T$ under condition $(G)_{\sigma_d,\sigma_\phi}^h$.

This last equation is an estimate on the CRT. Since $\Upsilon_{n,K}$ is a homeomorphism from $\T^{(n,K)}$ to $\mathfrak{T}^{(K)}$, it only remains to be shown that
\begin{equation}\label{eq:exact}
\lim_{K\to \infty} {\bf P}^{(K)}\Bigl[ \sup_{x\in \mathfrak{T}^{(K)}} \abs{\lambda^{(\mathfrak{T},K)}(\overrightarrow{\mathfrak T_x^{(K)}}) -\mu^{(\mathfrak{T},K)}(\overrightarrow{\mathfrak T_x^{(K)}})  }>\epsilon\Bigr] =0.
\end{equation}
and that
\begin{equation}\label{eq:exacth}
\lim_{K\to \infty} {\bf P}^{(K)}\Bigl[ \sup_{x\in \mathfrak{T}^{h,K}} \abs{\lambda^{(\mathfrak{T}^h,K)}(\overrightarrow{\mathfrak T_x^{h,K}}) -\mu^{(\mathfrak{T}^h,K)}(\overrightarrow{\mathfrak T_x^{h,K}})  }>\epsilon\Bigr] =0
\end{equation}
for condition $(G)_{\sigma_d,\sigma_\phi}^h$.

Let $\eta>0$ be arbitrary. It suffices to show that
\begin{equation}\label{eq:manolo}
{\bf P}^{(K)}\left[\sup_{x\in\frak{T}}\abs{\lambda^{(\mathfrak T,K)}(\overrightarrow{\frak{T}^{(K)}_x})-\mu^{(\mathfrak T,K)}(\overrightarrow{\frak{T}_x^{(K)}})}\leq \epsilon\right]\geq 1-\eta
\end{equation}
for all $K$ large enough.

Fixing $\epsilon>0$, there exists a (possibly random) $\delta^\ast>0$ such that 
\[
\sup_{x\in \frak{T}} \lambda_{\frak T}(B(x,\delta^\ast))\leq \epsilon/6.
\]
The display above can be easily deduced using the compactness of $\frak{T}$ and the fact that $\lambda_{\frak{T}}$ has no atoms.
Therefore there exists $\delta$ small enough such that
\begin{equation}\label{eq:noatoms}
{\mathbb P}\left[\sup_{x\in \frak{T}} \lambda_\frak T(B(x,\delta))\leq \epsilon/6\right]\geq 1-\frac{\eta}{16}
\end{equation}
On the one hand, it is easy to see from \eqref{a_delta_blabla1} that, for any $\eta>0$, there exists $K_1$ such that, for any $K\geq K_1$, the decomposition of $\frak{T}^{(K)}$ into non-overlapping line-segments $[a_i,b_i],i=1,\dots,2K-1$ satisfies 
\begin{equation}\label{eq:Ineedthis}
\mathbf{P}^{(K)}\left[\max_{i=1,\dots,2K-1}d_{\frak{T}}(a_i,b_i)\leq \delta/2\right]\geq 1-\frac{\eta}{32}
\end{equation}
   On the other hand, by item (ii) of Theorem 3 in \cite{Al1}, for any $\eta>0$ there exists $K_2$ such that for any $K\geq K_2$, 
 \[{\bf P}^{(K)}\left[\max_{x\in\frak{T}}d_{\frak{T}}(x,\pi^{\frak T^{(K)}}(x))\leq \delta/2\right]\geq 1-\frac{\eta}{32}.\]
 Let $K_3=K_1\vee K_2$.
 By the two displays above, we have that 
 \[{\bf P}^{(K)}\left[\{x:\pi^{\frak T^{(K_3)}}(x)\in[a_i,b_i]\}\subseteq B(a_i,\delta), \quad \forall i=1,\dots,2K_3-1\right]\geq 1-\frac{\eta}{16}.
 \] 
 Therefore, by \eqref{eq:noatoms}
 \begin{equation}\label{eq:gini}
 {\bf P}^{(K)}\left[\mu^{(\mathfrak T,K_3)}([a_i,b_i])\leq \epsilon/6,\quad \forall i=1,\dots, 2K_3-1\right]\geq 1-\frac{\eta}{8}.
 \end{equation}
  
  Furthermore, by item (ii) of Theorem 3 in \cite{Al1},  $\lambda^{(\mathfrak T,K)}$ (when regarded as a measure over $\frak{T}$) converges almost surely in distribution to $\lambda_\frak T$ as $K\to\infty$.
  Therefore, since $\{x:\pi^{\frak T^{(K_3)}}(x)\in[a_i,b_i]\}$ is a closed set, 
  \[\limsup_{K\to\infty}\lambda^{(\frak T,K)}(\{x:\pi^{\frak T^{(K_3)}}(x)\in[a_i,b_i]\})\leq\lambda_\frak T(\{x:\pi^{\frak T^{(K_3)}}(x)\in[a_i,b_i]\}),\]
  for all $i=1,\dots,2K_3-1$
Hence, by \eqref{eq:gini}, there exists  $K^*$ such that for $K>K^*$,
  \begin{equation}\label{eq:gini2}
{\bf P}^{(K)}\left[\max_{i=1,\dots,2K_3-1}\lambda^{(\frak T,K)}(\{x:\pi^{\frak T^{(K_3)}}(x)\in[a_i,b_i]\})\leq \epsilon/3\right]\geq 1-\frac{\eta}{4}.
 \end{equation}
 
By the convergence in distribution of $\lambda^{(\frak T,K)}$ towards $\lambda_\frak T$, and the fact that the boundary of $\frak{T}_{a_i}$ is $a_i$ and $\lambda_\frak T(a_i)=0$,
\[
\lim_{K\to\infty}\max_{i=1,\dots,2K_3-1} \abs{\lambda^{(\frak T,K)}(\overrightarrow{\frak{T}}_{a_i})-\lambda_\frak T(\overrightarrow{\frak{T}}_{a_i})}=0.
\]
Therefore, for $K$ large enough we have that
\begin{equation}\label{eq:hailmary}
{\bf P}^{(K)}\left[\max_{i=1,\dots,2K_3-1}\abs{\lambda^{(\frak T,K)}(\overrightarrow{\frak{T}}_{a_i})-\lambda_\frak T(\overrightarrow{\frak{T}}_{a_i})}\geq \epsilon/3\right]\geq 1-\frac{\eta}{8}.
\end{equation}
Let $x\in \frak{T}$, let $i\in\{1,\dots, 2K_3-1\}$ be the index such that $\pi^{\frak T^{(K_3)}}(x)\in[a_i,b_i]$. We will discern according to whether $x\in\frak{T}^{(K_3)}$ or not. Assume $x\in\frak{T}^{(K_3)}$. We write
\begin{equation}\label{eq:triangle}
\begin{aligned}
&\sup_{x\in\frak{T}^{(K_3)}}\abs{\lambda^{(\frak T,K)}(\overrightarrow{\frak{T}}_x)-\mu^{(\frak T,K)}(\overrightarrow{\frak{T}}_x)}
\leq\sup_{x\in\frak{T}^{(K_3)}}\abs{\lambda^{(\frak T,K)}(\overrightarrow{\frak{T}}_x)-\lambda^{(\frak T,K)}(\overrightarrow{\frak{T}}_{a_i})}\\+&\sup_{x\in\frak{T}^{(K_3)}}\abs{\lambda^{(\frak T,K)}(\overrightarrow{\frak{T}}_{a_i})-\mu^{(\frak T,K)}(\overrightarrow{\frak{T}}_{a_i})}+
\sup_{x\in\frak{T}^{(K_3)}}\abs{\mu^{(\frak T,K)}(\overrightarrow{\frak{T}}_{a_i})-\mu^{(\frak T,K)}(\overrightarrow{\frak{T}}_{x})}.
\end{aligned}
\end{equation}
Moreover, since $x\in\frak{T}^{(K_3)}$,
\[\overrightarrow{\frak{T}}_{a_i}\setminus\overrightarrow{\frak{T}}_x\subseteq \{x:\pi^{\frak T^{(K_3)}}(x)\in[a_i,b_i]\}.\]
Furthermore, for any $x\in\frak{T}$ we have $\overrightarrow{\frak{T}_x}\subseteq\overrightarrow{\frak{T}_{a_i}}$. Therefore $\abs{\lambda^{(\frak T,K)}(\overrightarrow{\frak{T}}_x)-\lambda^{(\frak T,K)}(\overrightarrow{\frak{T}}_{a_i})}=\lambda^{(\frak T,K)}(\overrightarrow{\frak{T}_{a_i}}\setminus\overrightarrow{\frak{T}}_x)$ and $ \abs{\mu^{(\frak T,K)}(\overrightarrow{\frak{T}}_{a_i})-\mu^{(\frak T,K)}(\overrightarrow{\frak{T}}_{x})}=\mu^{(\frak T,K)}(\overrightarrow{\frak{T}_{a_i}}\setminus\overrightarrow{\frak{T}}_x)$.
Hence, from the display above, it follows that
\[\abs{\lambda^{(\frak T,K)}(\overrightarrow{\frak{T}}_x)-\lambda^{(\frak T,K)}(\overrightarrow{\frak{T}}_{a_i})}\leq \lambda^{(\frak T,K)}[\{x:\pi^{\frak T^{(K_3)}}(x)\in[a_i,b_i]\}]\]
and
\[\abs{\mu^{(\frak T,K)}(\overrightarrow{\frak{T}}_x)-\mu^{(\frak T,K)}(\overrightarrow{\frak{T}}_{a_i})}\leq \mu^{(\frak T,K)}[\{x:\pi^{\frak T^{(K_3)}}(x)\in[a_i,b_i]\}]\]
Therefore, from displays \eqref{eq:gini} and \eqref{eq:gini2}, we get that the first and last summand of \eqref{eq:triangle} are smaller than $\epsilon/3$ (outside an event of probability smaller than $3/8\eta$). 
Therefore, by \eqref{eq:hailmary}, 
\begin{equation}\label{eq:pre}
{\bf P}^{(K)}\left[\sup_{x\in\frak{T}^{(K_3)}}\abs{\lambda^{(\frak T,K)}(\overrightarrow{\frak{T}}_x)-\mu^{(\frak T,K)}(\overrightarrow{\frak{T}}_{x})}\leq \epsilon\right]\geq 1-\eta/2.
\end{equation}

If $x\notin\frak{T}^{(K_3)}$, then, both $\frak{T}^{(K_3)}_x$ and $\frak{T}^{(K_3)}$ are contained in $\{x\in\frak{T}:\pi^{\frak T^{(K_3)}}(x)\in[a_i,b_i]\}$. Therefore
\begin{align*}
&\abs{\mu^{(\frak T,K)}(\overrightarrow{\frak{T}}_{x})-\lambda^{(\frak T,K)}(\overrightarrow{\frak{T}}_{x})}\\
\leq &\lambda^{(\frak T,K)}(\{x\in\frak{T}:\pi^{\frak T^{(K_3)}}(x)\in[a_i,b_i]\})\vee\mu^{(\frak T,K)}(\{x\in\frak{T}:\pi^{\frak T^{(K_3)}}(x)\in[a_i,b_i]\}).
\end{align*}
Finally, by \eqref{eq:gini} and \eqref{eq:gini2} we have that 
\[
{\bf P}^{(K)}\left[\sup_{x\notin\frak{T}^{(K_3)}}\abs{\lambda^{(\frak T,K)}(\overrightarrow{\frak{T}}_x)-\mu^{(\frak T,K)}(\overrightarrow{\frak{T}}_{x})}\leq \epsilon\right]\geq 1-\eta/2.
\]
We have established \eqref{eq:manolo} which is what we aimed to do.
 
 The proof for \eqref{eq:exacth} is completely analogous and that finishes the proof.

\end{proof}

\section{Application to lattice trees} \label{sect_lattice_trees}

In this section we apply Theorem~\ref{thm:height2} to the case of lattice trees. We will apply this theorem for points $(V_n^i)_{n\in \N}$ which are i.i.d.~uniformly chosen in the graphs $G_n$. 

The application of Theorem~\ref{thm:height2} will provide our main result on lattice trees (Theorem~\ref{main_thm}). For this, we need to show the following:
\begin{enumerate}
\item the points $(V_n^i)_{n\in \N}$ are chosen according to uniform edge-volume, 
\item condition $(G)^{h,+}_{1,\sigma_0,C_0}$ is verified, where $C_0$ is defined above equation (1.3) in~\cite{CFHP} and $\sigma_0$ is defined below (1.8) in the same article. 
\item condition $(S)$ is verified, 
\item condition $(R)_1$ is verified.
\end{enumerate}

The values of $C_1$ and $C_2$ appearing in Theorem~\ref{main_thm} will be
\[
C_1=  \sigma_0 \text{ and } C_2=C_0^{-1}.
\]

The remainder of this section is dedicated to proving the four points listed above.

\subsection{Uniformly chosen points are chosen according to uniform edge-volume}

Since $G_n$ is a lattice tree, and hence a tree,  we know that for $B\subset V(G_n)$ such that  $B^*$  is connected we know that $1+\abs{E(B^*)}=\abs{B}$ and that $1+\abs{E(G_n)}=\abs{G_n}$. In particular
\[
\frac{\abs{E(B^*)}}{\abs{E(G_n)}}=\frac{1+\abs{B}}{1+\abs{G_n}} \text{ and } \abs{\frac{1+\abs{B}}{1+\abs{G_n}}-\frac{\abs{B}}{\abs{G_n}}}\leq \frac 1{\abs{G_n}},
\]
now, since ${\mathbb P}[V_n^i\mid G_n]=\frac{\abs{B}}{\abs{G_n}}$ and $\abs{G_n}\geq n$ we know that for any $\epsilon>0$
\[
\lim_{n\to \infty}  {\mathbb P}\Bigl[\max_{\substack{B\subset V(G_n) \\ B^* \text{connected}}} \abs{{\mathbb P}[V_1 \in B\mid G_n] -\frac{\abs{E(B^*)}}{\abs{E(G_n)}} }>\epsilon\Bigr]=0.
\]

\subsection{Graphs are asymptotically tree-like}

In several of our conditions, we require that the graphs $\T^{(n,K)}$ are asymptotically tree-like with probability going to 1 (see Definition~\ref{def_thin}). Recalling Remark~\ref{skeleton_lat}, we see that $\T^{(n,K)}$  is always tree-like  for lattice trees.

\subsection{Condition $(G)^{h,+}$}

Considering Remark~\ref{skeleton_lat2} and Remark~\ref{skeleton_lat3}, we can see that the construction of the graph spatial tree $(T^{(n,K)},\sigma_{d} d_{T^{(n,K)}},\sigma_{\phi} \sqrt{\sigma_{d} } \phi_{T^{(n,K)}})$ is exactly the same as the construction of $\mathcal{B}_k(W)$ described in example 17 in the recent article~\cite{CFHP} (in that example the scaling factor $n$ was dropped from the notation for simplicity). Also, as pointed out in Remark \ref{rk:equivalence}, the corresponding continuous skeletons $\mathcal{B}_K(W)$ and $\mathfrak{B}^{(K)}$ have the same law.

With this in mind, we see that Remark~18 of said article (which uses as a key input the historical $k$-point function for lattice trees from~\cite{CFHP2}), shows that Condition $(G)^{h,+}_{1,\sigma_0,C_0}$ holds in our context (critical lattice trees conditioned on height, where the skeleton is built using vertices chosen according to the uniform measure).

\subsection{Condition $(S)$}

In~\cite{CFHP}, Theorem 19 (b) and (c) correspond to the second and third properties of Condition $(S)$. This means Condition $(S)$ is verified (since the \lq\lq tree-like\rq\rq~condition is always verified for lattice trees).

\subsection{Condition $(R)$}

Since $G_n$ is a tree, it is clear that $R^{G_n}_{\text{eff}}(\cdot,\cdot)=d_{G_n}(\cdot,\cdot)$ and thus condition $(R)_1$ is verified.

\bibliographystyle{plain}

\end{document}